\documentclass[10pt, reqno]{amsart}

\usepackage[utf8]{inputenc}
\usepackage{amsmath}
\usepackage{amssymb}
\usepackage{mathrsfs}
\usepackage[pdftex]{graphicx}
\usepackage{epstopdf}
\usepackage{geometry}
\usepackage{setspace}
\usepackage{color}
\usepackage{amsthm}
\usepackage{subfigure}
\usepackage{paralist}
\usepackage{dsfont}
\usepackage{verbatim}
\usepackage[numbers]{natbib}
\usepackage[breaklinks, colorlinks, linkcolor = black, citecolor = black, filecolor = black, urlcolor = black, bookmarksopen=true]{hyperref} 
\usepackage{float}

\theoremstyle{definition}
\newtheorem{definition}{Definition}[section]

\theoremstyle{plain}
\newtheorem{lemma}[definition]{Lemma}
\newtheorem{thm}[definition]{Theorem}
\newtheorem{prop}[definition]{Proposition}

\numberwithin{equation}{section}

\allowdisplaybreaks

\newcommand{\metr}{\mathbf{d}}

\newcommand{\ent}{\mathcal{E}}
\newcommand{\fc}{\mathcal{F}}

\newcommand{\dd}{\,\mathrm{d}}

\newcommand{\eps}{\varepsilon}
\newcommand{\W}{\mathbf{W}}

\newcommand{\X}{\mathbf{X}}

\newcommand{\N}{{\mathbb{N}}}
\newcommand{\R}{{\mathbb{R}}}

\newcommand{\eins}[1]{\mathbf{1}_{#1}}

\newcommand{\scrC}{\mathscr{C}}

\newcommand{\meas}{\mathscr{M}}

\newcommand{\Y}{\mathbf{Y}}

\newcommand{\argmin}{\operatornamewithlimits{argmin}}
\newcommand{\rec}{\mathrm{rec}}
\newcommand{\Nx}{{N_{\Delta x}}}
\newcommand{\Nt}{{N_{\Delta t}}}
\newcommand{\Dx}{\Delta x}
\newcommand{\Dt}{\Delta t}
\newcommand{\uD}{u^\Delta}
\newcommand{\wD}{w^\Delta}

\renewcommand{\bar}{\overline}
\renewcommand{\hat}{\widehat}

\begin{document}

\begin{abstract}
We propose a fully discrete variational scheme for nonlinear evolution equations with gradient flow structure on the space of finite Radon measures on an interval with respect to a generalized version of the Wasserstein distance with nonlinear mobility. Our scheme relies on a spatially discrete approximation of the semi-discrete (in time) minimizing movement scheme for gradient flows. Performing a finite-volume discretization of the continuity equation appearing in the definition of the distance, we obtain a finite-dimensional convex minimization problem usable as an iterative scheme. We prove that solutions to the spatially discrete minimization problem converge to solutions of the spatially continuous original minimizing movement scheme using the theory of $\Gamma$-convergence, and hence obtain convergence to a weak solution of the evolution equation in the continuous-time limit if the minimizing movement scheme converges. We illustrate our result with numerical simulations for several second- and fourth-order equations.
\end{abstract}

\title[Discrete approximation of the minimizing movement scheme]{Discrete approximation of the minimizing movement scheme for evolution equations of Wasserstein gradient flow type with nonlinear mobility}
\author[Jonathan Zinsl]{Jonathan Zinsl}
\address{Zentrum f\"ur Mathematik \\ Technische Universit\"at M\"unchen \\ 85747 Garching, Germany}
\email{jonathanzinsl@aol.com}
\author[Daniel Matthes]{Daniel Matthes}
\address{Zentrum f\"ur Mathematik \\ Technische Universit\"at M\"unchen \\ 85747 Garching, Germany}
\email{matthes@ma.tum.de}
\keywords{Gradient flow, minimizing movement scheme, modified Wasserstein distance, nonlinear mobility, spatial discretization, continuity equation, $\Gamma$-convergence}
\date{\today}
\subjclass[2010]{Primary: 35K52; Secondary: 35A15, 49J20, 65K10, 65M06}
\thanks{This research has been supported by the DFG Collaborative Research Center TRR 109 ``Discretization in Geometry and Dynamics''}

\maketitle

\section{Introduction}\label{sec:intro}
In this article, we introduce a fully discrete variational scheme for nonlinear evolution equations in one spatial dimension of the form
\begin{align}\label{eq:pde}
\partial_t u(t,x)&=\partial_x\left(m(u(t,x))\partial_x \frac{\delta\ent}{\delta u}(u(t,x))\right),
\end{align}
where $t>0$ and $x\in J$ with $J\subset\R$ an interval. Without loss of generality, we put $J=[0,1]$. Our sought solutions to equation \eqref{eq:pde} are nonnegative, satisfy the no-flux and Neumann boundary conditions
\begin{align}\label{eq:bcd}
\partial_x u(t,1)&=0=\partial_x u(t,0),\\
m(u(t,1))\partial_x \frac{\delta\ent}{\delta u}(u(t,1))&=0=m(u(t,0))\partial_x \frac{\delta\ent}{\delta u}(u(t,0)),
\end{align}
for all $t>0$, as well as the initial condition
\begin{align}\label{eq:icd}
u(0,\cdot)=u_0,
\end{align}
for a $u_0$ to be specified more in detail below.

Various second- and fourth-order evolution equations of the form \eqref{eq:pde} have been interpreted as gradient flows in spaces of measures w.r.t. the $L^2$-Wasserstein distance or its generalized versions \cite{dns2009, lisini2010, zm2014}, see for instance \cite{jko1998, otto2001, blanchet2008, gianazza2009, matthes2009, lisini2009, savare2008, lisini2012, loibl2015, zinsl2016, lisini2016, zinsl2016b}. There, the cornerstone in the proof of existence of (weak) solutions is the \emph{minimizing movement scheme} \cite{jko1998}, a time-discrete variational problem in a suitably chosen metric space $(\X,\metr)$, serving as a time-discrete approximation of the respective solution: given a suitable initial datum $u_0$ and a (small) step size $\tau>0$, define a sequence $(u_\tau^k)_{k\in\N}$ recursively by $u_\tau^0=u_0$ and
\begin{align}\label{eq:mms}
u_\tau^k\in\argmin_{u\in\X}\left(\frac1{2\tau}\metr^2(u,u_\tau^{k-1})+\ent(u)\right),\quad\text{for }k\in\N.
\end{align}
For the systems in the above-mentioned references, the free energy $\ent$, the mobility $m$ and the initial datum $u_0$ are such that the the piecewise constant (in time) interpolation $u_\tau$ along the sequence $(u_\tau^k)_{k\in\N}$ converges to a weak solution to the respective evolution equation \eqref{eq:pde} as $\tau\searrow 0$, in the sense stated in condition (MMS) below.

In this work, we make use of this property to set up a numerical scheme for \eqref{eq:pde}. In order to preserve the structural properties of the gradient flow, we do not fully discretize equation \eqref{eq:pde} itself, but spatially discretize the variational scheme \eqref{eq:mms}. Naturally, the most involved task there is to introduce a suitable discrete surrogate of the metric space $(\X,\metr)$.

\subsection{Mobilities and functionals: main assumptions}\label{subsec:assump}
For the mobility function $m:\R_{\ge 0}\to \R_{\ge 0}$, we distinguish two different cases, depending on the support of $m$. Let $M\in\R_{>0}\cup\{+\infty\}$. We always require
\begin{align}\label{eq:M}
\tag{M}
\begin{split}
m&\in C^2(0,M);\\
m(0)&=\lim_{z\searrow 0}m(z)=0,~\text{and if $M<\infty$,}~m(M)=\lim_{z\nearrow M}m(z)=0;\\
m(z)&>0~\text{and}~m''(z)\le 0~\text{for all}~ z\in (0,M).
\end{split}
\end{align}
With \eqref{eq:M}, one can endow the space $\X=\meas^+(J)$ of positive Radon measures on $J$ with the distance $\W_m$ from \cite{dns2009,lisini2010}:
\begin{align}
\W_m(\hat u,\check u)=\inf\bigg\{\int_0^1\Phi(u(t),w(t))\dd t:~(u,w)\in\scrC,~u(0)=\hat u,~u(1)=\check u\bigg\}^{1/2},
\end{align}
where the set $\scrC$ and the \emph{action functional} $\Phi$ are defined as follows (see \cite{dns2009,lisini2010} for more details).

For a given set $A\subset\R^d$, $\meas^+(A)$ and $\meas(A)$ denote the space of positive and signed Radon measures on $A$, respectively. Clearly, if $A$ is compact, elements in $\meas^+(A)$ and $\meas(A)$ are finite measures. Writing $I:=[0,1]$, $\scrC$ is the set of all pairs $(u,w)\in[\meas(I\times J)]^2$, where $(u(t))_{t\in I}$ is a Borel-measurable family in $\meas^+(J)$ and $(w(t))_{t\in I}$ is a Borel-measurable family in $\meas(J)$, such that the \emph{continuity equation} $\partial_t u=-\partial_x w$ (with the no-flux boundary condition $w(t,1)=0=w(t,0)$ for all $t>0$) is satisfied in the sense of distributions: for all $\varphi\in C^1(I\times J)$, one has
\begin{align}\label{eq:contdist}
&-\int_I\int_J\partial_t(t,x)\varphi \dd u(t,x)-\int_I\int_J \partial_x \varphi(t,x)\dd w(t,x)+\int_J\varphi(1,x)\dd u(1)(x)-\int_J\varphi(0,x)\dd u(0)(x)=0.
\end{align}
For the definition of the action functional $\Phi$, we first define the action density $\phi:\R_{\ge 0}\times \R\to\R_{\ge 0}\cup\{+\infty\}$ by
\begin{align}\label{eq:adens}
\phi(z,v)&=\begin{cases}\frac{v^2}{m(z)}&\text{if }z\in(0,M),\\ 0&\text{if }z\in\partial(0,M)~\text{and}~v=0,\\ +\infty&\text{else},\end{cases}
\end{align}
and recall \cite{dns2009} that $\phi$ is convex and lower semicontinuous. Thus, its recession function $\phi^\rec:\R_{\ge 0}\times \R\to\R\cup\{+\infty\}$ can be defined as in \cite[Def. 2.32]{ambrosio2000}:
\begin{align}\label{eq:rec}
\phi^\rec(z,v)=\lim_{s\to\infty}\frac{\phi(z_0+sz,v_0+sv)-\phi(z_0,v_0)}{s},
\end{align}
for arbitrary $(z_0,v_0)$ such that $\phi(z_0,v_0)<\infty$.
Now, given $u\in\meas^+(J)$ and $w\in\meas(J)$, we set
\begin{align}\label{eq:afunc}
\Phi(u,w)=\int_J\phi(u^\ll,w^\ll)\dd x+\int_J\phi^\rec\left(\frac{\dd (u^\perp,w^\perp)}{\dd |(u^\perp,w^\perp)|}\right)\dd |(u^\perp,w^\perp)|,
\end{align}
where $u=u^\ll+u^\perp$ and $w=w^\ll+w^\perp$ are the Lebesgue decompositions of $u$ and $w$. By a slight abuse of notation, we frequently identify measures which are absolutely continuous w.r.t. the Lebesegue measure (e.g. $u^\ll$ and $w^\ll$) on a certain set with their corresponding Lebesgue density.

Our results cover both cases $M<\infty$ and $M=+\infty$ as well as evolution equations \eqref{eq:pde} of second and of fourth order. We present our assumptions on $m$ and $\ent$ in the following.

\subsubsection{Mobilities}
If $M=+\infty$, the mobility function $m$ is required to satisfy one of the following two conditions: either,
\begin{align}\label{eq:wmob}
\tag{W}
m(z)=\bar m z\quad\text{for some constant }\bar m>0,
\end{align}
i.e. $\W_m$ is a scalar multiple of the classical $L^2$-Wasserstein distance \cite{brenier2000}, or
\begin{align}\label{eq:slmob}
\tag{SL}
\lim_{z\to\infty}\frac{m(z)}{z}=0,
\end{align}
i.e. $m$ grows sublinarly at $+\infty$. The paradigmatic examples for sublinear mobilities are power functions $m(z)=Cz^\alpha$ for $\alpha \in(0,1)$ and $C>0$.

If $M<\infty$, no further restrictions are imposed. In this case, the paradigmatic examples are given by $m(z)=C z^{\alpha_1}(M-z)^{\alpha_2}$ for $\alpha_1,\alpha_2 \in(0,1]$ and $C>0$.

\subsubsection{Second-order equations}
We consider energy functionals $\ent:\meas^+(J)\to\R\cup\{\infty\}$ of the form 
\begin{align}\label{eq:E2}
\tag{E2}
\ent(u)=\int_J E(u(x))\dd x+\int_J V(x)\dd u(x),
\end{align}
when $E(u)\in L^1(J)$; otherwise, $\ent(u)=+\infty$. The internal energy density $E:\R_{\ge 0}\to\R_{\ge 0}$ is assumed to be nonnegative, convex and continuous at $0$ (hence, $E$ is continuous on $\R_{\ge 0}$). If $M=+\infty$, we further require the following growth conditions when $E$ is not identically $0$:
\begin{align}
&\lim_{z\to\infty}\frac{E(z)}{z}=+\infty,\label{eq:Egrowth}\\
&\exists\,C>0\quad\forall\,z_1,z_2\ge 0:~E(z_1+z_2)\le C(E(z_1)+E(z_2)+1).\label{eq:doubling}
\end{align}
The superlinear growth condition \eqref{eq:Egrowth} ensures lower semicontinuity and stability of absolute continuity under weak$\ast$-convergence in $\meas^+(J)$ \cite{ambrosio2000}. The doubling condition \eqref{eq:doubling}---satisfied in all analytically interesting settings---is of more technical nature. The external potential $V:J\to \R$ is assumed to be H\"older-continuous. Equation \eqref{eq:pde} thus has the form of a nonlinear Fokker-Planck type equation:
\begin{align*}
\partial_t u=\partial_x(m(u)\partial_x (E'(u)+\partial_x V)).
\end{align*}

\subsubsection{Fourth-order equations}
Here, we consider energy functionals $\ent:\meas^+(J)\to\R\cup\{\infty\}$ of the form 
\begin{align}\label{eq:E4}
\tag{E4}
\ent(u)=\int_J G(\partial_x u(x))\dd x+\int_J E(u(x))\dd x+\int_J V(x)\dd u(x),
\end{align}
when $u\in H^1(J)$ (again, set $\ent(u)=+\infty$ otherwise). The gradient-dependent part $G:\R\to\R_{\ge 0}$ of the density shall be nonnegative and uniformly convex---take e.g. $G(p)=\frac12 p^2$ which yields the classical Dirichlet energy. For $E$ and $V$, we assume continuity. Equation \eqref{eq:pde} then reads as
\begin{align*}
\partial_t u=-\partial_x(m(u)\partial_x^2 G'(\partial_x u))+\partial_x(m(u)\partial_x (E'(u)+\partial_x V)),
\end{align*}
which comprises e.g. the classical Cahn-Hilliard and thin film equations.

\subsection{Discretization}
We now introduce our spatially discretized version of the minimizing movement scheme \eqref{eq:mms} and present our main results. Recall that our starting point is a single step of the minimizing movement scheme \eqref{eq:mms}, that is, the minimization problem
\begin{align}\label{eq:mms2}
\frac1{2\tau}\W_m(\hat u,\check u)^2+\ent(\check u)\longrightarrow\min\quad\text{w.r.t. } \check u\in\meas^+(J).
\end{align}
We assume---in addition to the requirements from Section \ref{subsec:assump}---that $\tau>0$ and $\hat u\in\meas^+(J)$ are such that the problem \eqref{eq:mms2} above has at least one solution. Since, by definition,
\begin{align}
\W_m(\hat u,\check u)^2=\inf\bigg\{\int_I\Phi(u(t),w(t))\dd t:~(u,w)\in\scrC,~u(0)=\hat u,~u(1)=\check u\bigg\},
\end{align}
we can rephrase the minimizing movement problem \eqref{eq:mms2} equivalently as follows:
\begin{align}\label{eq:mmscomb}
\frac{1}{2\tau}\int_I \Phi(u(t),w(t))\dd t+\ent(u(1))\longrightarrow\min\quad\text{w.r.t. } (u,w)\in\scrC,~u(0)=\hat u.
\end{align}
A simplified version of this problem has already been studied in \cite{burger2010}, where only fluxes $w$ of the particular form $w(t)=u(0)v(t)$ have been considered.

In order to avoid vanishing densities $u$, we consider at first the following regularized version of \eqref{eq:mmscomb}: for $\eps\in (0,1)$, define the regularized action density $\phi_\eps:\R_{\ge 0}\times\R\to\R_{\ge 0}\cup\{\infty\}$ by
\begin{align}\label{eq:regadens}
\phi_\eps(z,v)&=\begin{cases}\frac{v^2+\eps}{m(z)}&\text{if }z\in(0,M),\\ +\infty&\text{else}.\end{cases}
\end{align}
Since $z\mapsto (m(z))^{-1}$ is convex and nonnegative on $(0,M)$ thanks to assumption \eqref{eq:M}, $\phi_\eps$ is convex, and also lower semicontinuous. One easily verifies that its recession function $\phi_\eps^\rec$ coincides with $\phi^\rec$. Thus, the associated action functional $\Phi_\eps$ is given by
\begin{align}\label{eq:regafunc}
\Phi_\eps(u,w)=\int_J\phi_\eps(u^\ll,w^\ll)\dd x+\int_J\phi^\rec\left(\frac{\dd (u^\perp,w^\perp)}{\dd |(u^\perp,w^\perp)|}\right)\dd |(u^\perp,w^\perp)|,
\end{align}
and our regularized minimization problem reads
\begin{align}\label{eq:mmscombeps}
\frac{1}{2\tau}\int_I \Phi_\eps(u(t),w(t))\dd t+\ent(u(1))\longrightarrow\min\quad\text{w.r.t. } (u,w)\in\scrC,~u(0)=\hat u^\eps,
\end{align}
where 
\begin{align*}
\hat u^\eps:=\begin{cases}\hat u+\eps&\text{if }M=+\infty,\\\hat u+\eps\left(1-\frac2{M}\hat u\right)&\text{if }M<\infty.\end{cases}
\end{align*}

The idea of our discretization is as follows. We discretize the continuity equation according to the finite difference scheme on $I\times J$ and consider the associated family of piecewise constant interpolations $(u^\Delta,w^\Delta)_\Delta$ as the variables with respect to which the functional in \eqref{eq:mmscombeps} is to be minimized.

To this end, let $\Nx\in\N$ and $\Nt\in\N$ be the number of equally-sized spatial and temporal subintervals for $I$ and $J$ and let $\Dx=\frac1{\Nx}$ and $\Dt=\frac1{\Nt}$ be the associated spatial and temporal step sizes, respectively. Thus, the spatio-temporal domain $I\times J$ where the continuity equation is to be solved is decomposed in $N=\Nx\Nt$ rectangles of area $\Dx\Dt$ (i.e., an equidistant lattice). A pair of values $(\uD_{i,j},\wD_{i,j})$ is assigned to each cell ($i\in\{1,\ldots,\Nt\}$, $j\in\{1,\ldots,\Nx\}$). We use the abbreviations
\begin{align*}
I_i&:=((i-1)\Dt,i\Dt]\quad\text{for }i\in\{2,\ldots,\Nt\},\\
I_1&:=[0,\Dt],\\
J_j&:=((j-1)\Dx,j\Dx]\quad\text{for }j\in\{2,\ldots,\Nx\},\\
J_1&:=[0,\Dx].
\end{align*}
A discrete surrogate of the initial-boundary value problem
\begin{align*}
\partial_t u&=-\partial_x w \quad\text{in }(0,1)^2,\\
w(t,1)&=0=w(t,0)\quad\text{for all }t>0,\\
u(0,\cdot)&=\hat u^\eps,
\end{align*}
is the following system of linear equations in $\R^N\times\R^N$:
\begin{align}\label{eq:CE}
\tag{CE}
\begin{split}
(\uD_{i,j}-\uD_{i-1,j})\Dx+(\wD_{i,j+1}-\wD_{i,j})\Dt&=0\quad\forall i\in\{2,\ldots,\Nt\},\,j\in\{1,\ldots,\Nx-1\},\\
(\uD_{i,\Nx}-\uD_{i-1,\Nx})\Dx+(\wD_{i,1}-\wD_{i,\Nx})\Dt&=0\quad\forall i\in\{2,\ldots,\Nt\},\\
(\uD_{1,j}-\hat u^{\eps,\Delta}_j)\Dx+(\wD_{1,j+1}-\wD_{1,j})\Dt&=0\quad\forall j\in\{1,\ldots,\Nx-1\},\\
(\uD_{1,\Nx}-\hat u^{\eps,\Delta}_\Nx)\Dx+(\wD_{1,1}-\wD_{1,\Nx})\Dt&=0,\\
\wD_{i,1}&=0\quad\forall i\in\{1,\ldots,\Nt\},
\end{split}
\end{align}
where the piecewise constant approximation $(\hat u^{\eps,\Delta}_j)_{j=1,\ldots,\Nx}$ of the initial condition $\hat u^\eps$ is defined as
\begin{align}\label{eq:appic}
\hat u^{\eps,\Delta}_j=\frac1{\Dx}\int_{J_j}\dd \hat u^\eps,\quad\forall j\in\{1,\ldots,\Nx\}.
\end{align}
The densities $\uD$ and $\wD$ are now defined via piecewise constant interpolation, that is
\begin{align}\label{eq:pwc}
\begin{split}
\uD(t,x)&=\uD_{i,j}\quad\text{if $t\in I_i$ and $x\in J_j$ for some $i\in \{2,\ldots,\Nt\}$, $j\in\{1,\ldots,\Nx\}$},\\
\uD(t,x)&=\uD_{1,j}\quad\text{if $t\in I_1\setminus\{0\}$ and $x\in J_j$ for some $j\in\{1,\ldots,\Nx\}$},\\
\uD(0,x)&=\hat u^{\eps,\Delta}_j\quad\text{if $x\in J_j$ for some $j\in\{1,\ldots,\Nx\}$},\\
\wD(t,x)&=\wD_{i,j}\quad\text{if $t\in I_i$ and $x\in J_j$ for some $i\in \{1,\ldots,\Nt\}$, $j\in\{1,\ldots,\Nx\}$},\\
\end{split}
\end{align}
For the functional, we first observe that
\begin{align*}
\int_I \Phi_\eps(\uD(t),\wD(t))\dd t&=\Dt\Dx\sum_{i=1}^\Nt\sum_{j=1}^\Nx \phi_\eps(\uD_{i,j},\wD_{i,j}),\\
\int_J E(\uD(1,x))\dd x&=\Dx\sum_{j=1}^\Nx E(\uD_{\Nt,j}).
\end{align*}
The discrete counterpart of the potential energy reads as
\begin{align*}
\int_J V^\Delta(x)\uD(1,x)\dd x=\Dx\sum_{j=1}^\Nx V^\Delta_j \uD_{\Nt,j},
\end{align*}
where $V^\Delta$ is the piecewise constant function with
\begin{align}
V^\Delta(x)=V^\Delta_j=V((j-1)\Dx)\quad\text{if $x\in J_j$ for some $j\in\{1,\ldots,\Nx\}$}.
\end{align}
If $\ent$ also depends on the derivative $\partial_x u(1)$ via $G$, we replace $u(1)$ by the piecewise \emph{affine} interpolant $\bar \uD(1)$ along the values $(\uD_{\Nt,j})_{j=1,\ldots,\Nx}$, i.e.:
\begin{align*}
\bar\uD(1,x)&=\uD_{\Nt,1},\quad\text{if $x\in \left[0,\frac12 \Dx\right]$,}\\
\bar\uD(1,x)&=\frac{\uD_{\Nt,j+1}-\uD_{\Nt,j}}{\Dx},\quad\text{if $x\in J_j+\frac12\Dx$ for some $j\in\{1,\ldots,\Nx-1\}$},\\
\bar\uD(1,x)&=\uD_{\Nt,\Nx},\quad\text{if $x\in \left(1-\frac12 \Dx,1\right]$.}
\end{align*}
The discrete version of the gradient-dependent energy is
\begin{align*}
\int_J G(\partial_x\bar\uD(1,x))\dd x=\Dx\sum_{j=1}^{\Nx-1}G\left(\frac{\uD_{\Nt,j+1}-\uD_{\Nt,j}}{\Dx}\right).
\end{align*}
We subsume the energetic parts in the new functional $\ent^\Delta$:
\begin{align}
\ent^\Delta(\uD(1))=\int_J E(\uD(1,x))\dd x+\int_J V^\Delta(x)\uD(1,x)\dd x+\int_J G(\partial_x\bar\uD(1,x))\dd x
\end{align}

Then, our spatial discretization of the minimization problem \eqref{eq:mmscombeps} reads
\begin{align}\label{eq:mmsdisc}
\frac1{2\tau}\int_I \Phi_\eps(\uD(t),\wD(t))\dd t+\ent^\Delta(\uD(1))\longrightarrow\min,\quad\text{w.r.t. $(\uD_{i,j},\wD_{i,j})_{i,j}$ satisfying \eqref{eq:CE}.}
\end{align}

In the following section, we prove that under the condition $\ent(\hat u)<\infty$, minimizers to \eqref{eq:mmsdisc} converge (up to subsequences) to minimizers of \eqref{eq:mmscombeps}, as $\Delta=(\Dt,\Dx)\to 0$, and the latter converge to minimizers of \eqref{eq:mmscomb}, as $\eps\searrow 0$, thus to minimizers for the original minimizing movement scheme \eqref{eq:mms}. Our strategy of proof is based on $\Gamma$-convergence. 

The convergence of minimizers yields the applicability of \eqref{eq:mmsdisc} as a numerical scheme for solving \eqref{eq:pde} in the weak sense, given the gradient flow approach via \eqref{eq:mms} produces weak solutions in the following sense:
\begin{itemize}
\item[(MMS)] For every vanishing sequence $(\tau_k)_{k\in\N}$, the time-discrete solution $u_{\tau_k}$ defined via the minimizing movement scheme \eqref{eq:mms} exists, and there exists a subsequence and a distributional solution $u:\R_{\ge 0}\to\meas^+(J)$ to \eqref{eq:pde} with $\ent(u(t))\le \ent(u_0)$ (and satisfying $u(t,x)\in [0,M]$ for a.e. $x\in J$ if $M<\infty$) for all $t\ge 0$, such that $u_{\tau_k}(t)$ converges weakly$\ast$ to $u(t)$ as $k\to\infty$, pointwise w.r.t. $t\ge 0$. 
\end{itemize}
For the precise conditions on $\ent$ and $m$ which imply (MMS), we refer to the respective article in the bibliography below. Usually, they are more restrictive than our general assumptions from Section \ref{subsec:assump}.

With \eqref{eq:mmsdisc}, we are able to approximate the discrete solution $u_\tau$ from the original minimizing movement scheme \eqref{eq:mms} by iterating \eqref{eq:mmsdisc} for fixed $\eps$ and $\Delta$: 

Given $u_0\in\meas^+(J)$ with $\ent(u_0)<\infty$, $\tau>0$, $\eps\in(0,1)$, $\Nt\in\N$ and $\Nx\in\N$, define the sequence $(u_\tau^{\eps,\Delta,k})_{k\in\N}$ recursively by $u_\tau^{\eps,\Delta,0}=u_0^{\eps,\Delta}$ (according to \eqref{eq:appic} for $u_0$ in place of $\hat u$), and $u_\tau^{\eps,\Delta,k}=\uD(1)$, where $(\uD,\wD)$ is a solution to the minimization problem \eqref{eq:mmsdisc} with 
\begin{align*}
\hat u=\begin{cases}u_\tau^{\eps,\Delta,k-1}-\eps&\text{if }M=+\infty,\\ \frac{u_\tau^{\eps,\Delta,k-1}-\eps}{1-\frac{2\eps}{M}}&\text{if }M<\infty,\end{cases}
\end{align*}
for each $k\in\N$.
With this sequence, we can define the \emph{fully discrete function} $u_\tau^{\eps,\Delta}$ by piecewise constant interpolation (like $u_\tau$ is constructed from $(u_\tau^k)_{k\in\N})$.

Our main theorem on the convergence of $u_\tau^{\eps,\Delta}$ as $\tau,\eps,\Delta\to 0$ reads as follows.

\begin{thm}[Convergence of the scheme \eqref{eq:mmsdisc}]\label{thm:conv}
Assume that $m$ and $\ent$ are as in Section \ref{subsec:assump} and that condition \textup{(MMS)} holds. Let $u_0\in\meas^+(J)$ with $\ent(u_0)<\infty$ (and $u_0(x)\in [0,M]$ for a.e. $x\in J$ if $M<\infty$) be given. 

Let sequences $\tau_k\to 0$, $\eps_l\to 0$ and $\Delta_n\to 0$ be given and define, for each $k\in\N$, the family of fully discrete functions $(u^{\eps_l,\Delta_n}_{\tau_k})_{(l,n)\in\N^2}$ iteratively via \eqref{eq:mmsdisc}. Then, there exist subsequences $(\tau_{k_h})_{h\in\N}$, $(\eps_{l_h})_{h\in\N}$ and $(\Delta_{n_h})_{h\in\N}$ and a weak solution $u:\R_{\ge 0}\to\meas^+(J)$ to \eqref{eq:pde} in the sense from \textup{(MMS)} such that for each $t\ge 0$, $u^{\eps_{l_h},\Delta_{n_h}}_{\tau_{k_h}}(t)\rightharpoonup^\ast u(t)$ as $h\to\infty$. 
\end{thm}

In principle, this variational scheme can also be extended to cover nonlocal terms in the free energy $\ent$, e.g. an interaction potential, as well as coupled systems (see \cite{zm2014}). Furthermore, at least for energies of the form \eqref{eq:E2}, it might be possible to generalize our ideas to the spatially multi-dimensional setting. This is postponed to future research.

\subsection{Related studies}
In the last years, several numerical approaches to Wasserstein-type gradient flows have been studied taking up the \emph{Lagrangian} point of view which is particularly useful for discretizing the classical Wasserstein distance because of the underlying optimal transport. Especially, the spatially one-dimensional case is even more exceptional since there, the Wasserstein distance between measures can be expressed as the plain $L^2$ distance between the corresponding inverse distribution functions. This property has been made use of various times in order to design fully discrete schemes on grounds of the minimizing movement scheme, see for instance \cite{kinderlehrer1999, cavalli2010, matthes2014, matthes2015, osberger2015, blanchet2008}, or by direct discretization of the evolution of the inverse distribution functions as in \cite{gosse2006}. Other Lagrangian-type schemes, also for multiple spatial dimensions, involve e.g. particle methods \cite{westdickenberg2010}, moving meshes \cite{budd1999} or discretization and evolution of optimal transport maps \cite{carrillomoll, junge2016}. However, in our case dealing with genuinely nonlinear mobility functions, approaches involving optimal transport or inverse distribution functions do not seem to be possible at the first glance. In contrast, schemes of \emph{Eulerian} type, such as finite volume methods, do neither rely on one-dimensionality of space nor on linearity of the mobility, but might not pass on the variational structure of the equation to the discretization. Structure-preserving finite difference, volume or element discretizations for Fokker-Planck type equations have been introduced e.g. in \cite{chf2007, bcf2012, bc2012, liu2014, carrillo2015, cances2016}, see also the references therein. Our approach mainly focusses on the discretization of the semi-discrete minimizing movement scheme when the mobility is nonlinear, a case which has seemingly not been considered up to now. Special classes of similar second-order equations with possibly nonlinear mobility have been fully discretized in \cite{bcf2012} using finite differences, but not relying on the particular metric gradient flow structure. For linear mobility, a method for approximating the Wasserstein distance via entropic regularization of optimal transport has been studied in \cite{peyre2015}. Without optimal transport theory available, the most appealing starting point for discretizations of the generalized Wasserstein distance $\W_m$ seems to be its very definition via the generalized Benamou-Brenier formula \cite{brenier2000}. A finite-element discretization employing a linearized version of the constraints appearing in the Benamou-Brenier formula for the classical Wasserstein distance has been introduced in \cite{burger2010}. Our method can be seen as a generalization: apart from allowing for nonlinear mobilities, we do not perform a linearization of the optimization problem for $\W_m$. Still, our method is structure-preserving, variational and completely elementary at its core, which makes it easy to be implemented. In contrast to \cite{bcf2012, burger2010}, we are also able to prove the convergence of our scheme (however, without specifying the rate). A numerical method relying on the Benamou-Brenier formula for the classical Wasserstein distance has also been introduced in the recent articles \cite{benamou2015,benamou2015b} where a combination of Galerkin or finite element methods for spatial discretization with an augmented Lagrangian method for solving the minimization problem in $\W_m$ has been investigated. Compared to \cite{benamou2015,benamou2015b}, our approach can be applied for a broader class of problems and seems to be more direct and less technical. Furthermore, our proof of convergence does neither rely on previous results on e.g. the convergence of finite element methods nor on certain regularity properties of the solution. Our scheme is applicable for a wide class of second---and notably also fourth---order evolution equations, amongst others Fokker-Planck type equations (possibly with nonlinear diffusion), the Cahn-Hilliard equation for phase separation and the thin film equation generating the Hele-Shaw flow.

\subsection{Outline of the paper}
Section \ref{sec:proof} is concerned with the proof of Theorem \ref{thm:conv}: first, we show a $\Gamma$-convergence property of the functionals associated with the minimization problems \eqref{eq:mmscomb}, \eqref{eq:mmscombeps} and \eqref{eq:mmsdisc} before the statement in Theorem \ref{thm:conv} is proved. In Section \ref{sec:num}, we illustrate the result with several numerical simulations covering the cases of linear and nonlinear mobilities as well as second- and fourth-order equations.

\section{Proof of convergence}\label{sec:proof}
Up to some diagonal arguments, Theorem \ref{thm:conv} follows from the convergence of solutions to the minimization problems \eqref{eq:mmscomb}, \eqref{eq:mmscombeps} and \eqref{eq:mmsdisc} as $\Delta\to 0$ and $\eps\to 0$. As a preparation to show $\Gamma$-convergence, we first introduce a suitable topological space before recalling the definition of the respective functionals more in detail.

Consider $\Y=\meas(I\times J)\times\meas(I\times J)$, endowed with the weak$\ast$ topology, and let $\tau>0$ and $\hat u\in\meas^+(J)$ such that $\ent(\hat u)<\infty$.

The functional to be minimized in \eqref{eq:mmscomb} is $\fc:\Y\to \R\cup\{\infty\}$ with
\begin{align}\label{eq:F}
\fc(u,w)=\frac1{2\tau}\int_I\int_J \phi(u^\ll,w^\ll)\dd x\dd t+\frac1{2\tau}\int_I\int_J \phi^\rec\left(\frac{\dd (u^\perp,w^\perp)}{\dd |(u^\perp,w^\perp)|}\right)\dd |(u^\perp,w^\perp)|(t,x)+\ent(u(1)),
\end{align}
for the Lebesgue decompositions $u=u^\ll+u^\perp$ and $w=w^\ll+w^\perp$, if $u(0)=\hat u$ and $(u,w)\in\scrC$; and $\fc(u,w)=\infty$ otherwise. 

Similarly, for all $\eps\in(0,1)$, let $\fc_\eps:\Y\to \R\cup\{\infty\}$ with
\begin{align}\label{eq:Feps}
\fc_\eps(u,w)=\frac1{2\tau}\int_I\int_J \phi_\eps(u^\ll,w^\ll)\dd x\dd t+\frac1{2\tau}\int_I\int_J \phi^\rec\left(\frac{\dd (u^\perp,w^\perp)}{\dd |(u^\perp,w^\perp)|}\right)\dd |(u^\perp,w^\perp)|(t,x)+\ent(u(1)),
\end{align}
for the Lebesgue decompositions $u=u^\ll+u^\perp$ and $w=w^\ll+w^\perp$, if $u(0)=\hat u^\eps$ and $(u,w)\in\scrC$; and $\fc_\eps(u,w)=\infty$ otherwise.

Finally, for all $\eps\in(0,1)$ and $\Nx,\Nt\in\N$, define $\fc_{\eps,\Delta}:\Y\to \R\cup\{\infty\}$ with
\begin{align}\label{eq:Fepsdel}
\fc_{\eps,\Delta}(u,w)=\frac1{2\tau}\int_I\int_J \phi_\eps(u,w)\dd x\dd t+\ent^\Delta(u(1)),
\end{align}
if $u$ and $w$ are given by piecewise constant interpolation via \eqref{eq:pwc}, and the corresponding family of values $(\uD_{i,j},\wD_{i,j})_{i,j}$ satisfies \eqref{eq:CE}; $\fc_{\eps,\Delta}(u,w)=\infty$ otherwise.

The main result of this section is
\begin{thm}[$\Gamma$-convergence and convergence of minimizers]\label{thm:gc}
Let $\hat u\in\meas^+(J)$ such that $\ent(\hat u)<\infty$ and $\tau>0$ be given. The following statements hold:
\begin{enumerate}[(a)]
\item For fixed $\eps\in(0,1)$, $\fc_{\eps,\Delta}\stackrel{\Gamma}{\rightharpoonup^\ast}\fc_\eps$ as $\Delta\to 0$.
\item For each $\eps\in(0,1)$ and $\Delta$, $\fc_{\eps,\Delta}$ possesses a minimizer $(u^{\eps,\Delta},w^{\eps,\Delta})$ on $\Y$ which is an element of the weakly$\ast$-compact set $K=\{(u,w)\in\Y|\int_I\int_J \dd u\le C,\,\int_I\int_J\dd |w|\le C\}$, where $C>0$ is a constant independent of $\eps$ and $\Delta$. Furthermore, for fixed $\eps\in(0,1)$, $(u^{\eps,\Delta},w^{\eps,\Delta})$ converges weakly$\ast$ (up to subsequences) as $\Delta\to 0$ to a minimizer $(u^\eps,w^\eps)\in K$ of $\fc_\eps$.
\item As $\eps\searrow 0$, $\fc_{\eps}\stackrel{\Gamma}{\rightharpoonup^\ast}\fc$.
\item If, for each fixed $\eps\in (0,1)$, $(u^\eps,w^\eps)\in K$ is a minimizer of $\fc_\eps$, there exists a subsequence such that $(u^\eps,w^\eps)$ converges weakly$\ast$ as $\eps\searrow 0$ to a minimizer $(u,w)\in K$ of $\fc$.
\end{enumerate}
\end{thm}

We divide the proof into smaller steps. For later reference, we summarize the following obvious results on the recession function $\phi^\rec$ in

\begin{lemma}[Recession function {\cite{dns2009,lisini2010}}]\label{lem:rec}~
\begin{enumerate}[(a)]
\item Assume that $M=\infty$ and $m$ satisfies \eqref{eq:wmob}. Then, the recession function $\phi^\rec$ of $\phi$ is given by 
\begin{align*}
\phi^\rec\equiv \phi.
\end{align*}
\item If $M=\infty$ and $m$ satisfies \eqref{eq:slmob}, one has
\begin{align*}
\phi^\rec(z,v)=\begin{cases}0&\text{if }v=0,\\
+\infty&\text{otherwise.}\end{cases}
\end{align*}
\item If $M<\infty$, then
\begin{align*}
\phi^\rec(z,v)=\begin{cases}0&\text{if }(z,v)=(0,0),\\
+\infty&\text{otherwise.}\end{cases}
\end{align*}
\item For all $\eps\in(0,1)$, one has
\begin{align*}
\phi^\rec_\eps\equiv \phi^\rec.
\end{align*}
\end{enumerate}
\end{lemma}

This result particularly allows us in the cases (b) and (c) to conclude absolute continuity of $w$ and $(u,w)$, respectively, if $\Phi(u,w)<\infty$. 

\subsection{\texorpdfstring{$\Gamma$-convergence of $\fc_{\eps,\Delta}$ as $\Delta\to 0$}{Gamma-convergence of the epsilon-Delta functional}}

We now address the first part of the proof of Theorem \ref{thm:gc}(a).

\begin{prop}[$\liminf$ estimate for $\fc_{\eps,\Delta}$]\label{prop:Fepsdelinf}
Fix $\eps\in(0,1)$ and assume that $(\uD,\wD)_{\Delta}$ converges weakly$\ast$ in $\Y$ to $(u,w)\in\Y$ as $\Delta=(\Dt,\Dx)\to 0$. Then, $\fc_\eps(u,w)\le\liminf\limits_{\Delta\to 0}\fc_{\eps,\Delta}(\uD,\wD)$.
\end{prop}

\begin{proof}
Without loss of generality, since $\fc_{\eps,\Delta}$ is bounded from below, we can assume that \break $\sup_{\Delta}\fc_{\eps,\Delta}(\uD,\wD)<\infty$. Hence, $(\uD,\wD)$ are piecewise constant densities on $I\times J$ with values satisfying \eqref{eq:CE}, for each $\Delta$. We seek to verify the weak formulation \eqref{eq:contdist} of the continuity equation for the limit $(u,w)$. To this end, we fix $\varphi\in C^1(I\times J)$ and define, for each $\Delta$, the piecewise constant function $\varphi^\Delta$ such that $\varphi^\Delta(t,x)=\varphi^\Delta_{i,j}=\varphi((i-1)\Dt,(j-1)\Dx)$ for $(t,x)\in I_i\times J_j$, $i\in\{1,\ldots,\Nt\}$, $j\in\{1,\ldots,\Nx\}$. Multiplication of the respective equation in \eqref{eq:CE} with $\varphi^\Delta_{i,j}$ and summation yields
\begin{align*}
\sum_{i=1}^\Nt\sum_{j=1}^\Nx(\varphi^\Delta_{i,j}(\uD_{i,j}-\uD_{i-1,j})\Dx+\varphi^\Delta_{i,j}(\wD_{i,j+1}-\wD_{i,j})\Dt)&=0,
\end{align*}
putting $\wD_{i,\Nx+1}=0$ for each $i$ and $\varphi^\Delta_{0,j}=0$ for all $j$, and writing $\hat u^{\eps,\Delta}_j=\uD_{0,j}$ for all $j$. Rearranging the sums yields
\begin{align*}
&\Dx\Dt\sum_{i=1}^{\Nt-1}\sum_{j=1}^\Nx \frac{\varphi^\Delta_{i,j}-\varphi^\Delta_{i+1,j}}{\Dt}\uD_{i,j}+\Dx\Dt \sum_{i=1}^\Nt
\sum_{j=1}^{\Nx-1}\frac{\varphi^\Delta_{i,j}-\varphi^\Delta_{i,j+1}}{\Dx}\wD_{i,j}\\
&+\Dx \sum_{j=1}^\Nx(\varphi^\Delta_{\Nt,j}\uD_{\Nt,j}-\varphi^\Delta_{1,j}\uD_{0,j})=0.
\end{align*}
We express in terms of integrals to obtain
\begin{align}\label{eq:dweak}
\begin{split}
&\int_{I\setminus I_{\Nt}}\int_J \frac{\varphi^\Delta(t,x)-\varphi^\Delta(t+\Dt,x)}{\Dt}\uD(t,x)\dd x\dd t\\&+\int_I\int_{J\setminus J_{\Nx}}\frac{\varphi^\Delta(t,x-\Dx)-\varphi^\Delta(t,x)}{\Dx}\wD(t,x)\dd x\dd t\\&+\int_J(\varphi^\Delta(1,x)\uD(1,x)-\varphi^\Delta(0,x)\uD(0,x))\dd x=0.
\end{split}
\end{align}
Since $\Dx\sum_{j=1}^\Nx \uD_{i,j}=\Dx\sum_{j=1}^\Nx \hat u^{\eps,\Delta}_{j}=\int_J\dd \hat u^\eps\le \int_J \dd \hat u+\eps<\infty$ for all $i$ and $\Delta$, weak$\ast$-convergence of $\uD$ to $u$ yields $\int_I\int_J\dd u(t,x)=\int_J \dd \hat u+\eps$, and on a suitable subsequence, $\uD(1)\rightharpoonup^\ast u(1)$. Passing to the limit $\Delta\to 0$ in \eqref{eq:dweak} clearly yields \eqref{eq:contdist}, using that the terms involving $\varphi$ converge uniformly. Notice that by construction, $\hat u^{\eps,\Delta}\rightharpoonup^\ast \hat u^\eps$ as $\Delta\to 0$.

Thus, we have shown that $\fc_\eps(u,w)<\infty$ if and only if $\int_I\Phi_\eps(u(t),w(t))\dd t$ and $\ent(u(1))$ are finite. Now, \cite[Thm. 2.34]{ambrosio2000} on weak$\ast$-lower semicontinuity of certain integral functionals yields
\begin{align*}
&\int_I\int_J \phi_\eps(u^\ll,w^\ll)\dd x\dd t+\int_I\int_J \phi^\rec\left(\frac{\dd (u^\perp,w^\perp)}{\dd |(u^\perp,w^\perp)|}\right)\dd |(u^\perp,w^\perp)|\le \liminf_{\Delta\to 0}\int_I\int_J\phi_\eps(\uD,\wD)\dd x\dd t,
\end{align*}
thanks to the convexity of $\phi_\eps$ and $\phi_\eps^\rec\equiv\phi^\rec$. 

Using the H\"older continuity of $V$, one easily sees that
\begin{align*}
\lim_{\Delta\to 0}\int_J V^\Delta\uD(1,x)\dd x=\int_J V\dd u(1).
\end{align*}

In order to prove the $\liminf$ estimate for the internal energy, we distinguish the cases \eqref{eq:E2} and \eqref{eq:E4}.

Assume that $\ent$ is of the form \eqref{eq:E2}. If $M=+\infty$, convexity and superlinear growth \eqref{eq:Egrowth} imply with the help of \cite[Ex. 2.36]{ambrosio2000} that $E(u(1))\in L^1(J)$ and
\begin{align}\label{eq:Elsc}
\int_J E(u(1))\dd x&\le \liminf_{\Delta\to 0}\int_J E(\uD(1))\dd x.
\end{align}
If $M<\infty$, the family $(\uD(1))_\Delta$ is bounded in all $L^p(J)$. Again, convexity implies \eqref{eq:Elsc} via Alaoglu's theorem (extracting a subsequence if necessary).

Consider now the case of gradient-dependent energy density \eqref{eq:E4}. Since $\sup_{\Delta}\ent(\uD(1))<\infty$, $\|\bar\uD(1)\|_{H^1(J)}$ is uniformly bounded w.r.t. $\Delta$. Hence, on a suitable subsequence, since $H^1(J)$ is compactly contained in $C^{1/2}(J)$, $\bar \uD(1)\to u(1)\in C^{1/2}(J)$ uniformly by the Arzel\`{a}-Ascoli theorem. Uniform H\"older continuity of $(\bar\uD(1))_\Delta$ also implies that $\uD(1)\to u(1)$ uniformly. Alaoglu's theorem implies that $\partial_x \bar \uD(1)\rightharpoonup \partial_x u(1)$ in $L^2(J)$. Using the uniform convexity of $G$, the continuity of $E$ and the dominated convergence theorem, one has
\begin{align*}
\int_J G(\partial_x \bar \uD(1))\dd x&\le \liminf_{\Delta\to 0}\int_J G(\partial_x u(1))\dd x,\\
\lim_{\Delta \to 0}\int_J E(\uD(1))\dd x&=\int_J E(u(1))\dd x.
\end{align*}

All in all, the desired $\liminf$ estimate for $\fc_\eps$ follows.
\end{proof}

In order to complete the proof of Theorem \ref{thm:gc}(a), we show
\begin{prop}[Recovery sequence for $\fc_{\eps,\Delta}$]
Fix $\eps\in(0,1)$ and let $(u,w)\in\Y$ with $\fc_\eps(u,w)<\infty$ be given. Define, for each $\Delta$, piecewise constant functions $(\uD,\wD)$ according to \eqref{eq:pwc}, with values
\begin{align}\label{eq:rs}
\begin{split}
\uD_{i,j}&=\frac1{\Dt\Dx}\int_{I_i}\int_{J_j}\dd u,\quad\forall i\in\{2,\ldots,\Nt\},~j\in\{1,\ldots,\Nx\},\\
\uD_{1,j}&=\frac1{\Dt\Dx}\int_{(0,\Dt]}\int_{J_j}\dd u,\quad\forall j\in\{1,\ldots,\Nx\},\\
\wD_{i,j}&=\frac1{\Dt\Dx}\int_{I_i}\int_{J_j}\dd w,\quad\forall i\in\{1,\ldots,\Nt\},~j\in\{2,\ldots,\Nx\},\\
\wD_{i,1}&=0,\quad\forall i\in\{1,\ldots,\Nt\}.
\end{split}
\end{align}
Then, $(\uD,\wD)\rightharpoonup^\ast (u,w)$ as $\Delta\to 0$, and
\begin{align}\label{eq:deltasup}
\limsup_{\Delta\to 0}\fc_{\eps,\Delta}(\uD,\wD)\le \fc_\eps(u,w).
\end{align}
\end{prop}

\begin{proof}
We first sketch that the definition of $(\uD,\wD)$ via \eqref{eq:rs} and \eqref{eq:pwc} leads to a density satisfying \eqref{eq:CE}. For instance, in order to verify the first set of conditions in \eqref{eq:CE}, we take a sequence $(\varphi^\delta)_{\delta>0}$ in $C^1(I\times J)$ such that 
\begin{align*}
\partial_t\varphi^\delta\to \frac1{\Dt}(\eins{I_i\times J_j}-\eins{I_{i-1}\times J_j}),\\
\partial_x\varphi^\delta\to \frac1{\Dx}(\eins{I_i\times J_{j+1}}-\eins{I_i\times J_j}),
\end{align*}
pointwise in $I\times J$ as $\delta\to 0$. Since $(u,w)\in\scrC$ as $\fc_\eps(u,w)<\infty$, we get---using the dominated convergence theorem---that
\begin{align*}
0&=\int_I\int_J \frac1{\Dt}(\eins{I_i\times J_j}-\eins{I_{i-1}\times J_j})\dd u+\int_I\int_J\frac1{\Dx}(\eins{I_i\times J_{j+1}}-\eins{I_i\times J_j})\dd w,
\end{align*}
which is by construction of $(\uD,\wD)$ nothing else as
\begin{align*}
(\uD_{i,j}-\uD_{i-1,j})\Dx+(\wD_{i,j+1}-\wD_{i,j})\Dt&=0.
\end{align*}
The remaining conditions in \eqref{eq:CE} can be considered similarly.

We now prove that $\uD\rightharpoonup^\ast u$. The proof of $\wD\rightharpoonup^\ast w$ can be done similarly. Fix $f\in C(I\times J)$. By definition of $\uD$, one has
\begin{align*}
\int_I\int_J f \uD\dd x\dd t-\int_I\int_J f\dd u&=\sum_{i=1}^\Nt\sum_{j=1}^\Nx\int_{I_i}\int_{J_j}\left[\frac1{\Dt\Dx}\int_{I_i}\int_{J_j} f(s,y)\dd s\dd y-f(t,x)\right]\dd u(t,x)\\
&=\int_I\int_J (f^\Delta-f)\dd u,
\end{align*}
where $f^\Delta-f\to 0$ pointwise everywhere on $I\times J$ since $f$ is continuous (thus, every point $(t,x)\in I\times J$ is a Lebesgue point). Again, the dominated convergence theorem yields the asserted
\begin{align*}
\lim_{\Delta\to 0}\left(\int_I\int_J f \uD\dd x\dd t-\int_I\int_J f\dd u\right)&=0.
\end{align*}
We now treat all integrals appearing in $\fc_\eps$ separately and distinguish cases for the action part. Assume at first that $m$ satisfies $M=\infty$ and \eqref{eq:slmob}. Then, thanks to $\fc_\eps(u,w)<\infty$ and Lemma \ref{lem:rec}(b), we have $w=w^\ll$. By construction of $(\uD,\wD)$, one gets
\begin{align*}
&\int_I\Phi_\eps(\uD(t),\wD(t))\dd t=\int_I\int_J\phi_\eps(\uD,\wD)\dd x\dd t=\Dt\Dx\sum_{i=1}^\Nt\sum_{j=1}^\Nx \phi_\eps(\uD_{i,j},\wD_{i,j}).
\end{align*}
Taking into account that $\wD_{i,1}=0$, we have
\begin{align*}
\Dt\Dx\sum_{i=1}^\Nt\sum_{j=1}^\Nx \phi_\eps(\uD_{i,j},\wD_{i,j})&\le\Dt\Dx\sum_{i=2}^\Nt\sum_{j=1}^\Nx \phi_\eps\left(\frac1{\Dt\Dx}\int_{I_i}\int_{J_j}\dd u,\frac1{\Dt\Dx}\int_{I_i}\int_{J_j}\dd w\right)\\
&\quad+\Dt\Dx\sum_{j=1}^\Nx \phi_\eps\left(\frac1{\Dt\Dx}\int_{(0,\Dt]}\int_{J_j}\dd u,\frac1{\Dt\Dx}\int_{I_1}\int_{J_j}\dd w^\ll\right).
\end{align*}
Since $m$ is nondecreasing, one has for every Borel set $A\subset I\times J$:
\begin{align*}
m\left(\int_A \dd u\right)\ge m\left(\int_A \dd u^\ll\right).
\end{align*}
So, using the absolute continuity of $(u^\ll,w^\ll)$:
\begin{align*}
&\Dt\Dx\sum_{i=1}^\Nt\sum_{j=1}^\Nx \phi_\eps(\uD_{i,j},\wD_{i,j})\le \Dt\Dx \sum_{i=1}^\Nt\sum_{j=1}^\Nx \frac{\left(\frac1{\Dt\Dx}\int_{I_i}\int_{J_j}\dd w^\ll\right)^2+\eps}{m\left(\frac1{\Dt\Dx}\int_{I_i}\int_{J_j}\dd u^\ll\right)}
\end{align*}
We apply Jensen's inequality ($\phi_\eps$ is convex):
\begin{align*}
\Dt\Dx \sum_{i=1}^\Nt\sum_{j=1}^\Nx \frac{\left(\frac1{\Dt\Dx}\int_{I_i}\int_{J_j}\dd w^\ll\right)^2+\eps}{m\left(\frac1{\Dt\Dx}\int_{I_i}\int_{J_j}\dd u^\ll\right)}&\le \int_I\int_J\phi_\eps(u^\ll,w^\ll)\dd x\dd t=\int_I\Phi_\eps(u(t),w(t))\dd t.
\end{align*}
All in all, we have:
\begin{align}\label{eq:epsl}
&\limsup_{\Delta\to 0}\int_I\Phi_\eps(\uD(t),\wD(t))\dd t\le \int_I\Phi_\eps(u(t),w(t))\dd t,
\end{align}
for all $\Delta$. The same calculations---apart from the monotonicity argument which is not needed since $u^\perp=0$---show that \eqref{eq:epsl} also holds for $M<\infty$. We now consider the remaining case $M=\infty$ and \eqref{eq:wmob}. First,
\begin{align*}
&\int_I\Phi_\eps(\uD(t),\wD(t))\dd t=\int_I\int_J \phi(\uD,\wD)\dd x\dd t+\int_I\int_J\frac{\eps}{m(\uD)}\dd x\dd t.
\end{align*}
The last integral above can be estimated with Jensen's inequality as before:
\begin{align*}
\int_I\int_J\frac{\eps}{m(\uD)}\dd x\dd t\le \int_I\int_J\frac{\eps}{m((u^\ll)^\Delta)}\dd x\dd t\le \int_I\int_J\frac{\eps}{m(u^\ll)}\dd x\dd t.
\end{align*}
Second, one has by construction of $(\uD,\wD)$ and since $\phi$ is 1-homogeneous if $m$ satisfies \eqref{eq:wmob}:
\begin{align*}
&\int_I\int_J \phi(\uD,\wD)\dd x\dd t=\Dx\Dt\sum_{i=1}^\Nt\sum_{j=1}^\Nx \phi(\uD_{i,j},\wD_{i,j})\\
&\le \Dx\Dt \sum_{i=2}^\Nt\sum_{j=1}^\Nx \phi\left(\frac1{\Dx\Dt}\int_{I_i}\int_{J_j}\dd u,\frac1{\Dx\Dt}\int_{I_i}\int_{J_j}\dd w\right)\\
&\quad +\Dx\Dt \sum_{j=1}^\Nx \phi\left(\frac1{\Dx\Dt}\int_{(0,\Dt]}\int_{J_j}\dd u,\frac1{\Dx\Dt}\int_{(0,\Dt]}\int_{J_j}\dd w\right)\\
&=\sum_{i=2}^\Nt\sum_{j=1}^\Nx \phi\left(\int_{I_i}\int_{J_j}\frac{\dd (u,w)}{\dd |(u,w)|}\dd |(u,w)|\right)+\sum_{j=1}^\Nx \phi\left(\int_{(0,\Dt]}\int_{J_j}\frac{\dd (u,w)}{\dd |(u,w)|}\dd |(u,w)|\right)\\
&=\sum_{i=2}^\Nt\sum_{j=1}^\Nx \phi\left(\int_{I_i}\int_{J_j}\frac{\dd (u,w)}{\dd |(u,w)|}\left(\int_{I_i}\int_{J_j}\dd |(u,w)|\right)\dd \frac{|(u,w)|}{\int_{I_i}\int_{J_j}\dd |(u,w)|}\right)\\
&\quad +\sum_{j=1}^\Nx \phi\left(\int_{(0,\Dt]}\int_{J_j}\frac{\dd (u,w)}{\dd |(u,w)|}\left(\int_{(0,\Dt]}\int_{J_j}\dd |(u,w)|\right)\dd \frac{|(u,w)|}{\int_{(0,\Dt]}\int_{J_j}\dd |(u,w)|}\right)\\
\end{align*}
Now, Jensen's inequality yields
\begin{align*}
&\sum_{i=2}^\Nt\sum_{j=1}^\Nx \phi\left(\int_{I_i}\int_{J_j}\frac{\dd (u,w)}{\dd |(u,w)|}\left(\int_{I_i}\int_{J_j}\dd |(u,w)|\right)\dd \frac{|(u,w)|}{\int_{I_i}\int_{J_j}\dd |(u,w)|}\right)\\
&\quad +\sum_{j=1}^\Nx \phi\left(\int_{(0,\Dt]}\int_{J_j}\frac{\dd (u,w)}{\dd |(u,w)|}\left(\int_{(0,\Dt]}\int_{J_j}\dd |(u,w)|\right)\dd \frac{|(u,w)|}{\int_{(0,\Dt]}\int_{J_j}\dd |(u,w)|}\right)\\
&\le \sum_{i=2}^\Nt\sum_{j=1}^\Nx \int_{I_i}\int_{J_j} \phi\left(\frac{\dd (u,w)}{\dd |(u,w)|}\left(\int_{I_i}\int_{(J_j}\dd |(u,w)|\right)\right)\dd \frac{|(u,w)|}{\int_{I_i}\int_{J_j}\dd |(u,w)|}\\
&\quad + \sum_{j=1}^\Nx \int_{(0,\Dt]}\int_{J_j} \phi\left(\frac{\dd (u,w)}{\dd |(u,w)|}\left(\int_{(0,\Dt]}\int_{(J_j}\dd |(u,w)|\right)\right)\dd \frac{|(u,w)|}{\int_{(0,\Dt]}\int_{J_j}\dd |(u,w)|}\\
&\le \sum_{i=2}^\Nt\sum_{j=1}^\Nx \int_{I_i}\int_{J_j} \phi\left(\frac{\dd (u,w)}{\dd |(u,w)|}\right)\dd |(u,w)|+\sum_{j=1}^\Nx \int_{(0,\Dt]}\int_{J_j} \phi\left(\frac{\dd (u,w)}{\dd |(u,w)|}\right)\dd |(u,w)|\\
&\le\int_I\int_J \phi\left(\frac{\dd (u,w)}{\dd |(u,w)|}\right)\dd |(u,w)|.
\end{align*}

By \cite[Prop. 2.37]{ambrosio2000}, we have
\begin{align*}
&\int_I\int_J\phi\left(\frac{\dd (u,w)}{\dd |(u,w)|}\right)\dd|(u,w)|\\
&=\int_I\int_J\phi\left(\frac{\dd (u^\ll,w^\ll)}{\dd |(u^\ll,w^\ll)|}\right)\dd|(u^\ll,w^\ll)|+\int_I\int_J\phi\left(\frac{\dd (u^\perp,w^\perp)}{\dd |(u^\perp,w^\perp)|}\right)\dd|(u^\perp,w^\perp)|\\
&=\int_I\int_J\phi\left(u^\ll,w^\ll\right)\dd x \dd t+\int_I\int_J\phi\left(\frac{\dd (u^\perp,w^\perp)}{\dd |(u^\perp,w^\perp)|}\right)\dd|(u^\perp,w^\perp)|,
\end{align*}
and arrive at
\begin{align*}
\int_I\int_J \phi(\uD,\wD)\dd x\dd t&\le \int_I\int_J\phi\left(u^\ll,w^\ll\right)\dd x \dd t+\int_I\int_J\phi\left(\frac{\dd (u^\perp,w^\perp)}{\dd |(u^\perp,w^\perp)|}\right)\dd|(u^\perp,w^\perp)|.
\end{align*}
Thus, we again have \eqref{eq:epsl}.

For the energetic part in $\fc_{\eps,\Delta}$, we distinguish between $\ent$ satisfying \eqref{eq:E2} and \eqref{eq:E4}, respectively. For $\ent$ of the form \eqref{eq:E2}, we can use Jensen's inequality once more to obtain
\begin{align*}
\int_J E(\uD(1,x))\dd x&\le \int_J E(u(1,x))\dd x.
\end{align*}
Clearly, we have
\begin{align*}
\lim_{\Delta\to 0}\int_J V^\Delta \uD(1)\dd x=\int_J V\dd u(1).
\end{align*}
Hence, we have
\begin{align}\label{eq:entlimsup}
\limsup_{\Delta\to 0}\ent^\Delta(\uD(1))\le \ent(u(1)),
\end{align}
which proves \eqref{eq:deltasup} for \eqref{eq:E2}.

Consider now gradient-dependent energies \eqref{eq:E4}. For all $\Delta$, we get using the definition of $(\uD,\wD)$:
\begin{align*}
\int_J G(\partial_x\bar\uD(1,x))\dd x&=\Dx\sum_{j=1}^{\Nx-1}G\left(\frac1{(\Dx)^2}\left(\int_{J_{j+1}}u(1,x)\dd x-\int_{J_j}u(1,x)\dd x\right)\right)\\
&=\Dx\sum_{j=1}^{\Nx-1}G\left(\frac1{\Dx}\int_{J_j}\frac{u(1,x+\Dx)-u(1,x)}{\Dx}\dd x\right)\\
&=\Dx\sum_{j=1}^{\Nx-1}G\left(\frac1{\Dx}\int_{J_j}\frac1{\Dx}\int_x^{x+\Dx}\partial_y u(1,y)\dd y\dd x\right).
\end{align*}
Recalling that $G$ is convex, we use Jensen's inequality twice:
\begin{align*}
&\Dx\sum_{j=1}^{\Nx-1}G\left(\frac1{\Dx}\int_{J_j}\frac1{\Dx}\int_x^{x+\Dx}\partial_y u(1,y)\dd y\dd x\right)\le \int_{J\setminus {J_{\Nx}}}\frac1{\Dx}\int_x^{x+\Dx}G(\partial_y u(1,y))\dd y\dd x,
\end{align*}
which is finite by continuity of the integrand w.r.t. $x$. With Fubini's theorem and elementary calculations, one sees that
\begin{align*}
\int_{J\setminus {J_{\Nx}}}\frac1{\Dx}\int_x^{x+\Dx}G(\partial_y u(1,y))\dd y\dd x=\int_J G(\partial_y u(1,y)) g^\Delta(y)\dd y,
\end{align*}
where
\begin{align*}
g^\Delta(y):=\begin{cases}\frac{y}{\Dx}&\text{if }y\in [0,\Dx),\\ 1&\text{if }y\in[\Dx,1-\Dx],\\ \frac{1-y}{\Dx}&\text{if }y\in(1-\Dx,1].\end{cases}
\end{align*}
Obviously, $g^\Delta\le 1$ and $g^\Delta\to 1$ pointwise on $(0,1)$. Thus, by dominated convergence,
\begin{align*}
\int_J G(\partial_y u(1,y)) g^\Delta(y)\dd y\to \int_J G(\partial_y u(1,y))\dd y
\end{align*}
as $\Delta\to 0$. In summary, we have
\begin{align*}
\limsup_{\Delta\to 0}\int_J G(\partial_x\bar\uD(1,x))\dd x\le  \int_J G(\partial_x u(1,x))\dd x<\infty.
\end{align*}
Especially, the family $(\bar \uD)_\Delta$ is bounded in $H^1(J)$, so, by extracting a uniformly convergent subsequence, one has
\begin{align*}
\lim_{\Delta\to 0}\left(\int_J E(\uD(1))\dd x+\int_J V^\Delta \uD(1)\dd x\right)&=\int_J E(u(1))\dd x+\int_J V\dd u(1),
\end{align*}
so \eqref{eq:entlimsup} and thus \eqref{eq:deltasup} also hold for \eqref{eq:E4}.
\end{proof}

\subsection{\texorpdfstring{Minimizers of $\fc_{\eps,\Delta}$}{Minimizers of the epsilon-Delta functional}}
This section is concerned with the proof of Theorem \ref{thm:gc}(b). The main argument is the following estimate on the action functional: with Jensen's inequality,
\begin{align*}
\int_I\int_J \phi_\eps(u,w)\dd x\dd t&\ge \phi_\eps\left(\int_I\int_J u\dd x\dd t,\int_I\int_J |w|\dd x\dd t \right),
\end{align*}
when $\fc_{\eps,\Delta}(u,w)<\infty$. Since by \eqref{eq:CE}, one has $\int_I\int_J u\dd x\dd t=\int_J \dd \hat u^\eps$, so there exists $C>0$ independent of $\eps\in (0,1)$ such that $m(\int_I\int_J u\dd x\dd t)\le C$, which in turn yields
\begin{align}\label{eq:L1w2}
\left(\int_I\int_J|w|\dd x\dd t\right)^2\le C\int_I\int_J \phi_\eps(u,w)\dd x\dd t.
\end{align}
With this estimate at hand, the existence of minimizers for $\fc_{\eps,\Delta}$ now follows by compactness (recall that if $\fc_{\eps,\Delta}(u,w)<\infty$, $(u,w)$ can be identified with the respective family of values $(\uD_{i,j},\wD_{i,j})$ in $\R^N\times\R^N$---hence, the problem is finite-dimensional).

Denote by $(u^{\eps,\Delta}_{\min},w_{\min}^{\eps,\Delta})$ a minimizer of $\fc_{\eps,\Delta}$. Then, $\int_I\int_J u_{\min}^{\eps,\Delta}\dd x\dd t\le\int_J \dd \hat u+1$ uniformly in $\eps$ and $\Delta$. By boundedness of below of $\fc_{\eps,\Delta}$ and \eqref{eq:L1w2}, there exists $D>0$ such that
\begin{align*}
\left(\int_I\int_J |w_{\min}^{\eps,\Delta}|\dd x\dd t\right)^2&\le D(1+\fc_{\eps,\Delta}(u^{\eps,\Delta}_{\min},w_{\min}^{\eps,\Delta}))\le D(1+\fc_{\eps,\Delta}(\hat u^{\eps,\Delta},0)),
\end{align*}
where $(\hat u^{\eps,\Delta},0)$ is to be understood as constant w.r.t. $t\in I$. Now, 
\begin{align*}
\fc_{\eps,\Delta}(\hat u^{\eps,\Delta},0)=\ent(\hat u^{\eps,\Delta}),
\end{align*}
which is bounded uniformly for small $\Delta$ and $\eps$, recall that $\ent(\hat u)<\infty$ by assumption. All in all, there exists $C>0$ independent of $\eps$ and $\Delta$ such that
$(u_{\min}^{\eps,\Delta},w_{\min}^{\eps,\Delta})\in K=\{(u,w)\in\Y|\int_I\int_J \dd u\le C,\,\int_I\int_J\dd |w|\le C\}$, which is weakly$\ast$-compact, as asserted. The remaining claim on the convergence of minimizers of $\fc_{\eps,\Delta}$ to minimizers of $\fc_{\eps}$ now immediately follows from $\Gamma$-convergence (see Theorem \ref{thm:gc}(a)) and \cite[Thm. 1.21]{braides2002}.\hfill\qed

\subsection{\texorpdfstring{$\Gamma$-convergence of $\fc_{\eps}$ as $\eps\to 0$}{Gamma-convergence of the epsilon functional}}
In this section, we prove Theorem \ref{thm:gc}(c)\&(d).
\begin{prop}[$\liminf$ estimate for $\fc_\eps$]
If $(u^\eps,w^\eps)_{\eps\in(0,1)}$ is a sequence in $\Y$ weakly$\ast$-converging to $(u,w)$ as $\eps\searrow 0$, then $\fc(u,w)\le\liminf\limits_{\eps\to 0}\fc_{\eps}(u^\eps,w^\eps)$.
\end{prop}
\begin{proof}
Without loss of generality, $\sup\limits_\eps\fc_\eps(u^\eps,w^\eps)<\infty$. Easily, one sees that $\hat u^\eps\rightharpoonup^\ast \hat u$ as $\eps\to 0$. Extracting a subsequence where $u_\eps(1)\rightharpoonup^\ast u(1)$, we obtain that $(u,w)\in\scrC$. Now, the proof is completed by using weak$\ast$-lower semicontinuity:
\begin{align*}
\fc(u,w)&\le \liminf_{\eps\to 0}\int_I \Phi(u^\eps(t),w^\eps(t))\dd t+\liminf_{\eps\to 0}\ent(u^\eps(1))\\&\le \liminf_{\eps\to 0}\int_I \Phi_\eps(u^\eps(t),w^\eps(t))\dd t+\liminf_{\eps\to 0}\ent(u^\eps(1))\le \liminf_{\eps\to 0}\fc_\eps(u^\eps,w^\eps),
\end{align*}
since $\phi(z,v)\le \phi_\eps(z,v)$.
\end{proof}

\begin{prop}[Recovery sequence for $\fc_\eps$]
Let $(u,w)\in\Y$ with $\fc(u,w)<\infty$ be given, and define, for sufficiently small $\eps\in(0,1)$, the measure $(u^\eps,w^\eps)\in\Y$ by
\begin{align*}
u^\eps&=\begin{cases}u+\eps&\text{if }M=\infty,\\ u+\eps\left(1-\frac2{M}u\right)&\text{if }M<\infty,\end{cases}\\
w^\eps&=\begin{cases}w&\text{if }M=\infty,\\ \left(1-\frac{2\eps}{M}\right)w&\text{if }M<\infty.\end{cases}\\
\end{align*}
Then, $(u^\eps,w^\eps)\rightharpoonup^\ast (u,w)$ as $\eps\to 0$ and
\begin{align}\label{eq:epslim}
\lim_{\eps\to 0}\fc_\eps(u^\eps,w^\eps)=\fc(u,w).
\end{align}
\end{prop}
\begin{proof}
Obviously, $u^\eps(0)=\hat u^\eps$ and $(u^\eps,w^\eps)\in\scrC$ for all $\eps\in(0,1)$, since $\fc(u,w)<\infty$, and $(u^\eps,w^\eps)\rightharpoonup^\ast (u,w)$ as $\eps\to 0$. We first consider the case $M=\infty$. There, we have
\begin{align*}
\int_I\Phi_\eps(u^\eps(t),w^\eps(t))\dd t=\int_I\int_J \phi_\eps(u^\ll+\eps,w^\ll)\dd x\dd t+\int_I\int_J \phi^\rec\left(\frac{\dd (u^\perp,w^\perp)}{\dd |(u^\perp,w^\perp)|}\right)\dd |(u^\perp,w^\perp)|.
\end{align*}
We show that
\begin{align}\label{eq:philim}
\lim_{\eps\to 0}\int_I\int_J \phi_\eps(u^\ll+\eps,w^\ll)\dd x\dd t=\int_I\int_J \phi(u^\ll,w^\ll)\dd x\dd t.
\end{align}
Clearly, the integrand converges pointwise. Moreover, since $m$ is nondecreasing, concave and strictly increasing at $0$, there exists $C>0$ such that $m(\eps)\ge C\eps$ for all $\eps\in (0,1)$ and we obtain:
\begin{align*}
\phi_\eps(u^\ll+\eps,w^\ll)=\frac{(w^\ll)^2+\eps}{m(u^\ll+\eps)}\le \frac{(w^\ll)^2}{m(u^\ll)}+\frac{\eps}{m(\eps)}\le \frac{(w^\ll)^2}{m(u^\ll)}+\frac1{C},
\end{align*}
which is integrable (recall that $\fc(u,w)<\infty$). The dominated convergence theorem now yields \eqref{eq:philim}. In the case $M<\infty$, we argue similarly and assume that there exists $\bar z\in\left[ \frac{M}{2},M\right)$ such that $m'(\bar z)=0$ (the case $\bar z\in\left(0,\frac{M}{2}\right)$ can be treated in complete analogy). Recall that since $M<\infty$ and $\fc(u,w)<\infty$, we have $(u^\perp,w^\perp)=0$ and $u(t,x)\in[0,M]$ almost everywhere. Since $m$ is concave, one obtains the following elementary bounds on $m(u^\eps)$:
\begin{align*}
m(u^\eps)\ge \begin{cases}\max(m(u),m(\eps))&\text{if }u\in\left[0,\frac{M}{4}\right),\\ \min\limits_{z\in\left[\frac{M}{4},\frac{M+\bar z}{2}\right]}m(z)&\text{if }u\in\left[\frac{M}{4},\frac{M+\bar z}{2}\right],\\ \max(m(u),m(M-\eps))&\text{if }u\in\left(\frac{M+\bar z}{2},M\right].\end{cases}
\end{align*}
Hence, using that $(w^\eps)^2\le w^2$, one has:
\begin{align*}
\phi_\eps(u^\eps,w^\eps)&\le\phi_\eps(u^\eps,w)\\&\le \phi(u,w)+\frac{\eps}{m(\eps)}+\frac{\eps}{m(M-\eps)}+\frac{w^2+\eps}{\min\limits_{z\in\left[\frac{M}{4},\frac{M+\bar z}{2}\right]}m(z)}\\
&\le \phi(u,w)\left(1+\frac{\max\limits_{z\in[0,M]}m(z)}{\min\limits_{z\in\left[\frac{M}{4},\frac{M+\bar z}{2}\right]}m(z)}\right)+\eps\left(\frac{1}{m(\eps)}+\frac{1}{m(M-\eps)}+\frac1{\min\limits_{z\in\left[\frac{M}{4},\frac{M+\bar z}{2}\right]}m(z)}\right),
\end{align*}
and the dominated convergence theorem yields \eqref{eq:philim} also in this case.

For the energetic part, we first observe that
\begin{align*}
\lim_{\eps\to 0}\int_J V \dd u_\eps(1)=\int_J V\dd u(1)
\end{align*}
by weak$\ast$ convergence (extracting a subsequence if necessary). If $M<\infty$, one easily gets
\begin{align}\label{eq:Erec}
\lim_{\eps\to 0}\int_J E(u^\eps(1))\dd x=\int_J E(u(1))\dd x
\end{align}
by bounded convergence, using the continuity of $E$ and $\|u^\eps\|_{L^\infty(J)}\le M$. For the gradient-dependent part (if present), there is nothing to prove since $\partial_x u^\eps=\partial_x u$ on $J$ for all $\eps\in(0,1)$.

If $M=+\infty$, we have $\|u_\eps(1)-u(1)\|_{L^1(J)}\to 0$ thanks to superlinear growth \eqref{eq:Egrowth}. Using the doubling condition \eqref{eq:doubling}, we get
\begin{align*}
E(u^\eps(1))\le C(1+E(u(1))+E(\eps)),
\end{align*}
and the r.h.s. is integrable by assumption. Hence, dominated convergence yields \eqref{eq:Erec}. If a gradient-dependent part is present, we have $\partial_x u^\eps(1)=\left(1-\frac{2\eps}{M}\right)\partial_x u(1)$, so $\partial_x u^\eps(1)\to\partial_x u(1)$ strongly in $L^2(J)$ as $\eps\to 0$, since $G(\partial_x u)\in L^1(J)$ implies $\partial_x u\in L^2(J)$ by uniform convexity of $G$. On a subsequence, we have $\partial_x u^\eps(1)\to\partial_x u(1)$ pointwise a.e. in $J$. Hence, by continuity of $G$, also $G(\partial_x u^\eps(1))\to G(\partial_x u(1))$ pointwise almost everywhere. Now, convexity of $G$ yields
\begin{align*}
G(\partial_x u^\eps(1))\le \left(1-\frac{2\eps}{M}\right)G(\partial_x u(1))+\frac{2\eps}{M}G(0)\le G(\partial_x u(1))+\frac{2}{M}G(0),
\end{align*}
which is integrable. Again, the dominated convergence theorem gives
\begin{align*}
\lim_{\eps\to 0}\int_J G(\partial_x u^\eps(1))\dd x&=\int_J G(\partial_x u(1))\dd x.
\end{align*}
All in all, we have shown that $\lim\limits_{\eps\to 0}\ent(u^\eps)=\ent(u)$, and \eqref{eq:epslim} follows.
\end{proof}

The remaining part (d) in Theorem \ref{thm:gc} now is an immediate consequence of the parts (b)\&(c), using \cite[Thm. 1.21]{braides2002}.

\subsection{Convergence of the iterative scheme}
This section is concerned with the proof of Theorem \ref{thm:conv} which follows by Theorem \ref{thm:gc} and some diagonal arguments. As a preparation, notice that since $C(J)$ is separable, a family $(u_k)_{k\in\N}$ in $\meas^+(J)$ with uniformly bounded total mass converges weakly$\ast$ to some $u\in\meas^+(J)$ if and only if
\begin{align*}
\lim_{k\to\infty}\int_J f_r \dd u_k =\int_J f_r\dd u,
\end{align*}
for all $r\in\N$, where $(f_r)_{r\in\N}$ is dense in $C(J)$.

Let sequences $\tau_k\to 0$, $\eps_l\to 0$ and $\Delta_n\to 0$ be given and denote by $u^{\eps_l,\Delta_n}_{\tau_k,\hat u}$ a minimizer of $\fc_{\eps_l,\Delta_n}$ for $\tau=\tau_k$ and prescribed $\hat u$. Theorem \ref{thm:gc}(b) yields, for each $l\in\N$, $k\in\N$ and admissible $\hat u$, the existence of a (non-relabelled) subsequence $(\Delta_n)$ and a minimizer $u^{\eps_l}_{\tau_k,\hat u}$ of $\fc_{\eps_l}$ such that $u^{\eps_l,\Delta_n}_{\tau_k,\hat u}\rightharpoonup^\ast u^{\eps_l}_{\tau_k,\hat u}$ as $n\to\infty$.

Performing a diagonal argument, we see that for all $k\in\N$ and $\hat u$, there exists a subsequence $(\Delta_n)$ such that for all $l\in\N$, one has $u^{\eps_l,\Delta_n}_{\tau_k,\hat u}\rightharpoonup^\ast u^{\eps_l}_{\tau_k,\hat u}$ as $n\to\infty$. 

In particular, for all $k\in\N$, $r\in\N$ and $\hat u$, there exists a subsequence $(\Delta_{n_l})$ such that for all $l\in\N$:
\begin{align*}
\left|\int_J f_r \dd u^{\eps_l,\Delta_{n_l}}_{\tau_k,\hat u}-\int_J f_r\dd u^{\eps_l}_{\tau_k,\hat u}\right|&\le\frac1{l},
\end{align*}
where $(f_r)_{r\in\N}$ is dense in $C(J)$.

By Theorem \ref{thm:gc}(d), there exists for every $k\in\N$ and $\hat u$ a (non-relabelled) subsequence $(\eps_l)$ and a minimizer $u_{\tau_k,\hat u}$ of $\fc$ such that $u^{\eps_l}_{\tau_k,\hat u}\rightharpoonup^\ast u_{\tau_k,\hat u}$ as $l\to\infty$. Hence, there exists for all $k\in\N$, $r\in\N$ and $\hat u$ a subsequence $(\eps_l)$ and a further subsequence of $(\Delta_{n_l})$, such that
\begin{align*}
\lim_{l\to\infty}\int_J f_r \dd u^{\eps_l,\Delta_{n_l}}_{\tau_k,\hat u}-\int_J f_r\dd u_{\tau_k,\hat u}=0,
\end{align*}
as $l\to\infty$.

In particular, for all $r\in\N$ and $\hat u$, there exist subsequences $(\eps_{l_k})$ and $(\Delta_{n_k})$ such that for all $k\in\N$:
\begin{align*}
\left|\int_J f_r \dd u^{\eps_{l_k},\Delta_{n_k}}_{\tau_k,\hat u}-\int_J f_r\dd u_{\tau_k,\hat u}\right|&\le\frac1{k}.
\end{align*}

Since for each fixed $k\in\N$ the fully discrete function $u^{\eps_{l_k},\Delta_{n_k}}_{\tau_k}$ and the semi-discrete function $u_{\tau_k}$ are constant w.r.t. $t$ on the same subintervals of $\R_{\ge 0}$, two further diagonal arguments yield the existence of subsequences $(\eps_{l_k})$ and $(\Delta_{n_k})$ such that for all $k\in\N$, $t\ge 0$ and $r\in\N$:
\begin{align*}
\left|\int_J f_r \dd u^{\eps_{l_k},\Delta_{n_k}}_{\tau_k}(t)-\int_J f_r\dd u_{\tau_k}(t)\right|&\le\frac1{k}.
\end{align*}

By the assumption in Theorem \ref{thm:conv} on the convergence of the minimizing movement scheme \eqref{eq:mms}, there exists a subsequence $(\tau_{k_h})$ and a weak solution $u$ to \eqref{eq:pde} such that for all $t\ge 0$ and all $r\in\N$:
\begin{align*}
\lim_{h\to\infty}\int_J f_r \dd u_{\tau_{k_h}}(t) =\int_J f_r\dd u(t).
\end{align*}
Hence, there exist further subsequences $(\eps_{l_h})$ and $(\Delta_{n_h})$ such that for all $t\ge 0$ and $r\in\N$:
\begin{align*}
\lim_{h\to\infty}\int_J f_r \dd u^{\eps_{l_h},\Delta_{n_h}}_{\tau_{k_h}}(t) =\int_J f_r\dd u(t),
\end{align*}
completing the proof of Theorem \ref{thm:conv}.\hfill\qed

\section{Numerical simulation}\label{sec:num}
In this section, we illustrate our numerical scheme with several examples. The simulations have been performed with MATLAB using Newton's method for solving the Euler-Lagrange system associated with the convex minimization problem \eqref{eq:mmsdisc}.

\subsection{Fokker-Planck type equations}
This section is concerned with second-order equations of the form
\begin{align}\label{eq:fpe}
\partial_t u=\partial_x^2 u^q+\partial_x (u\partial_x V),
\end{align}
for $q\ge 1$, which fit into our framework with $m(z)=z$ (cf. \eqref{eq:wmob}) and
\begin{align*}
E(z)=\begin{cases}z\log z-z+1&\text{if $q=1$,}\\ \frac1{q-1}z^q&\text{if $q>1$,}\end{cases}
\end{align*}
see \eqref{eq:E2}. We particularly consider the cases $q=1$ (linear diffusion) and $q=2$ (quadratic diffusion of porous medium type). The confinement potential used here is quadratic,
\begin{align*}
V(x)=50\left(x-0.5\right)^2.
\end{align*}
For the simulation, we put $\Nt=2$, $\Nx=300$, $\tau=10^{-4}$ and $\eps=10^{-8}$ and the initial datum
\begin{align*}
u_0(x)=\cos(8\pi x)+1.
\end{align*}
Figures \ref{fig:he} and \ref{fig:pme} show the spatially discrete function $u_\tau^{\eps,\Delta}(t)$ constructed via the scheme \eqref{eq:mmsdisc} at different time points $t$ for $q=1$ and $q=2$, respectively.

\begin{figure}[h]
\centering
\subfigure[$t=0$.]{\includegraphics[width=0.49\textwidth]{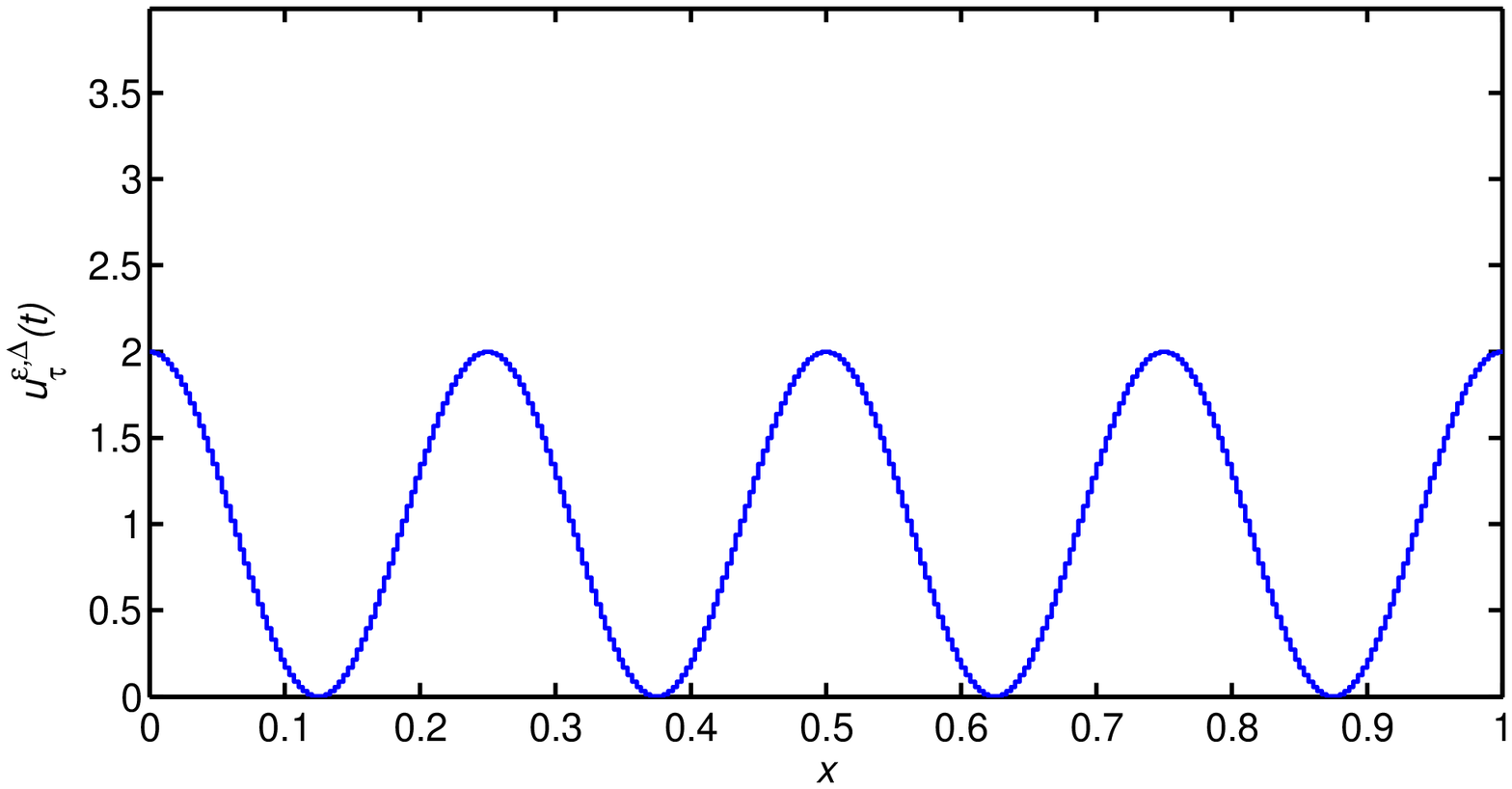}\label{he:a}}
\subfigure[$t=10\tau=10^{-3}$.]{\includegraphics[width=0.49\textwidth]{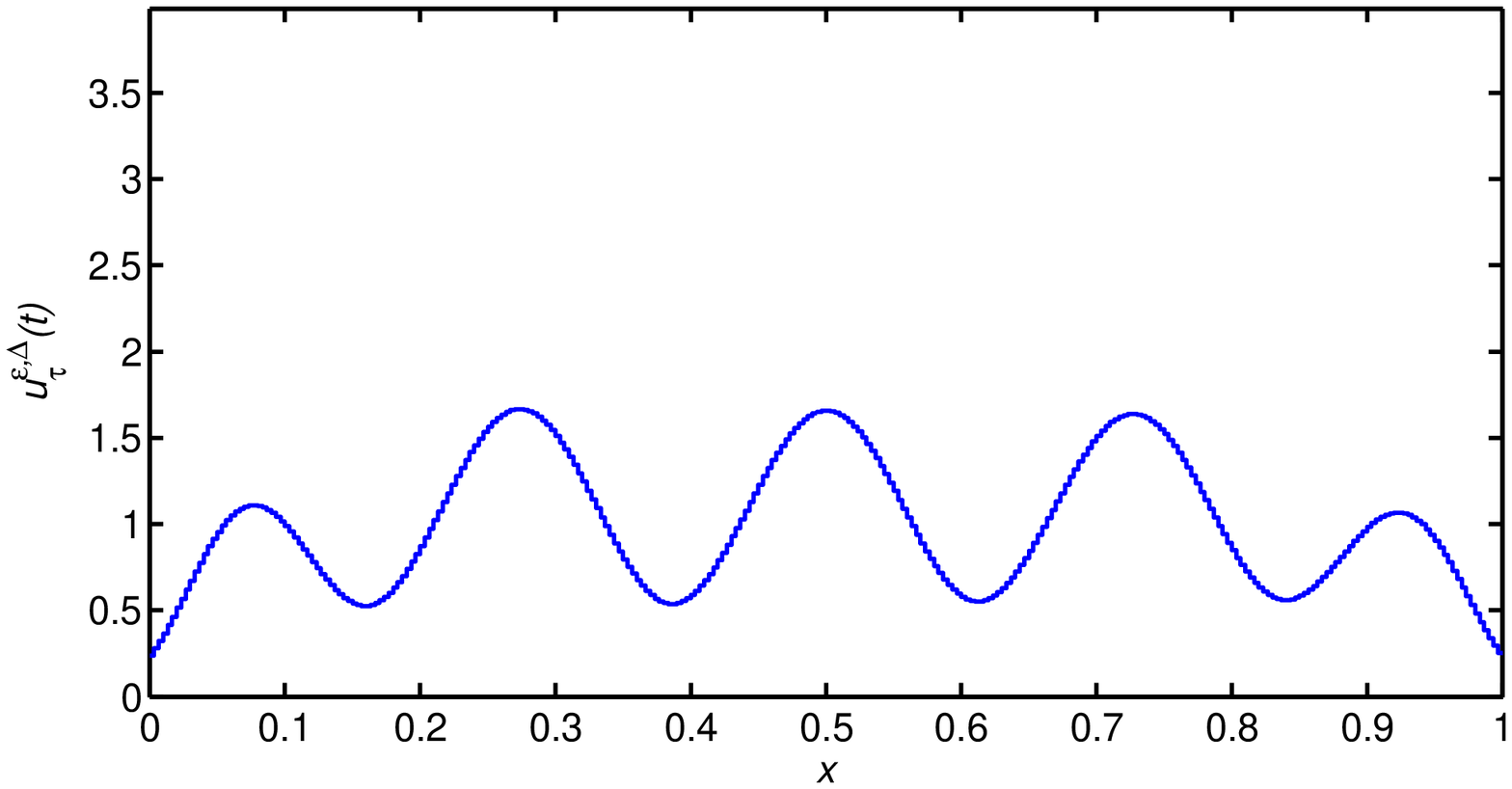}\label{he:b}}
\subfigure[$t=100\tau=0.01$.]{\includegraphics[width=0.49\textwidth]{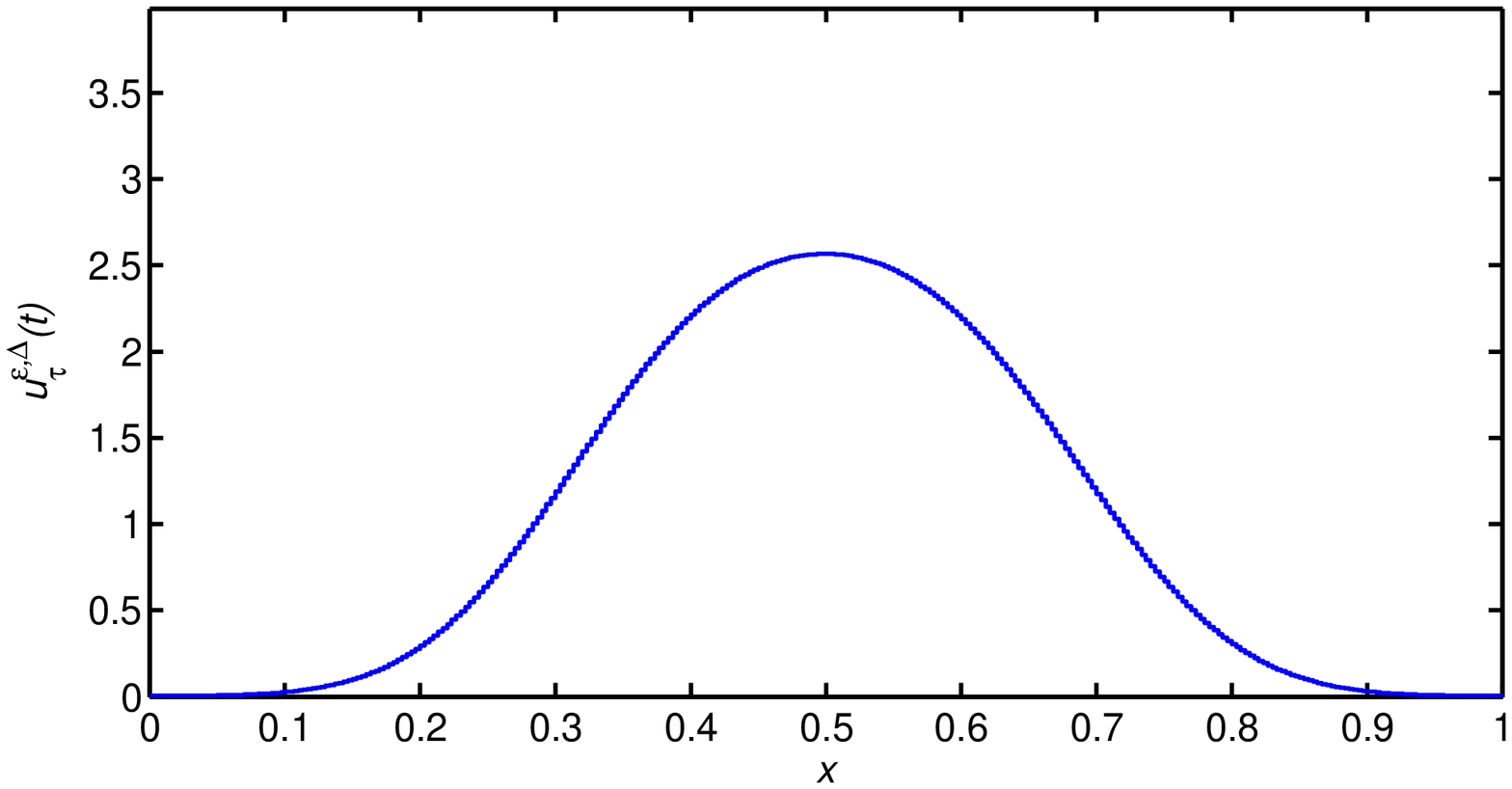}\label{he:c}}
\subfigure[$t=5000\tau=0.05$.]{\includegraphics[width=0.49\textwidth]{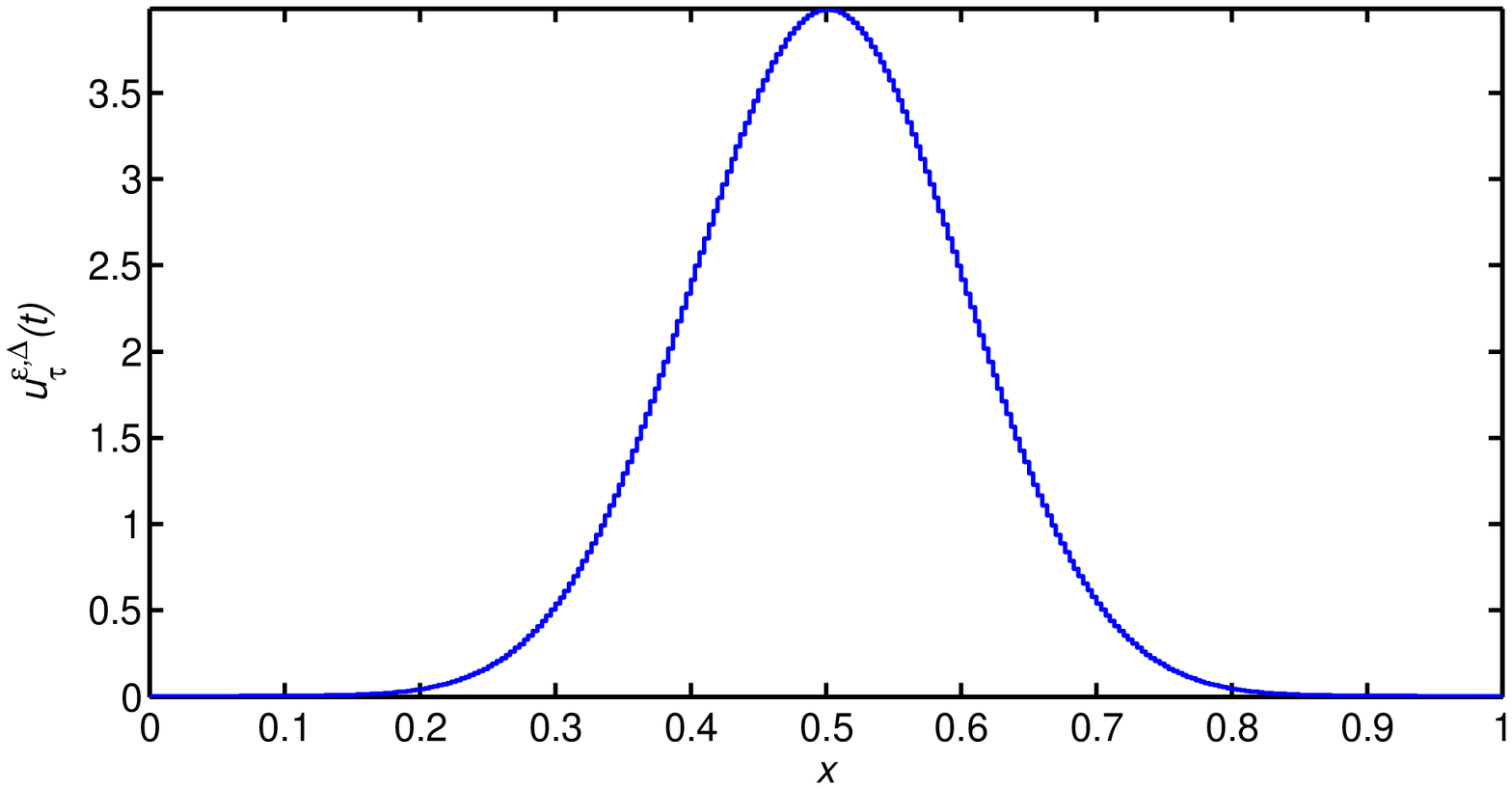}\label{he:d}}
        \caption{Numerical simulation of $u_\tau^{\eps,\Delta}(t)$ for \eqref{eq:fpe} with $q=1$.} 
        \label{fig:he}
\end{figure}

\begin{figure}[h]
\centering
\subfigure[$t=0$.]{\includegraphics[width=0.49\textwidth]{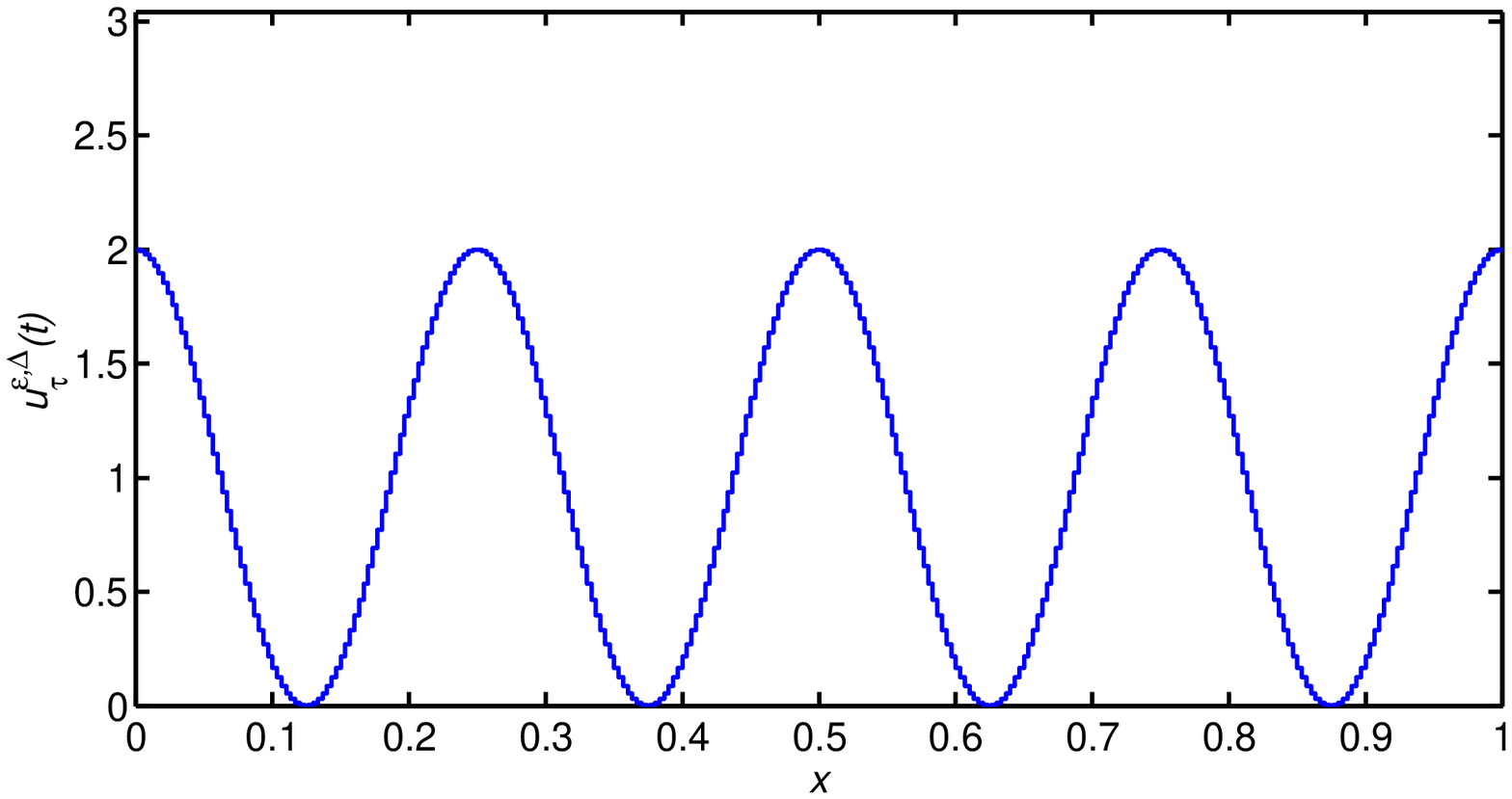}\label{pme:a}}
\subfigure[$t=10\tau=10^{-3}$.]{\includegraphics[width=0.49\textwidth]{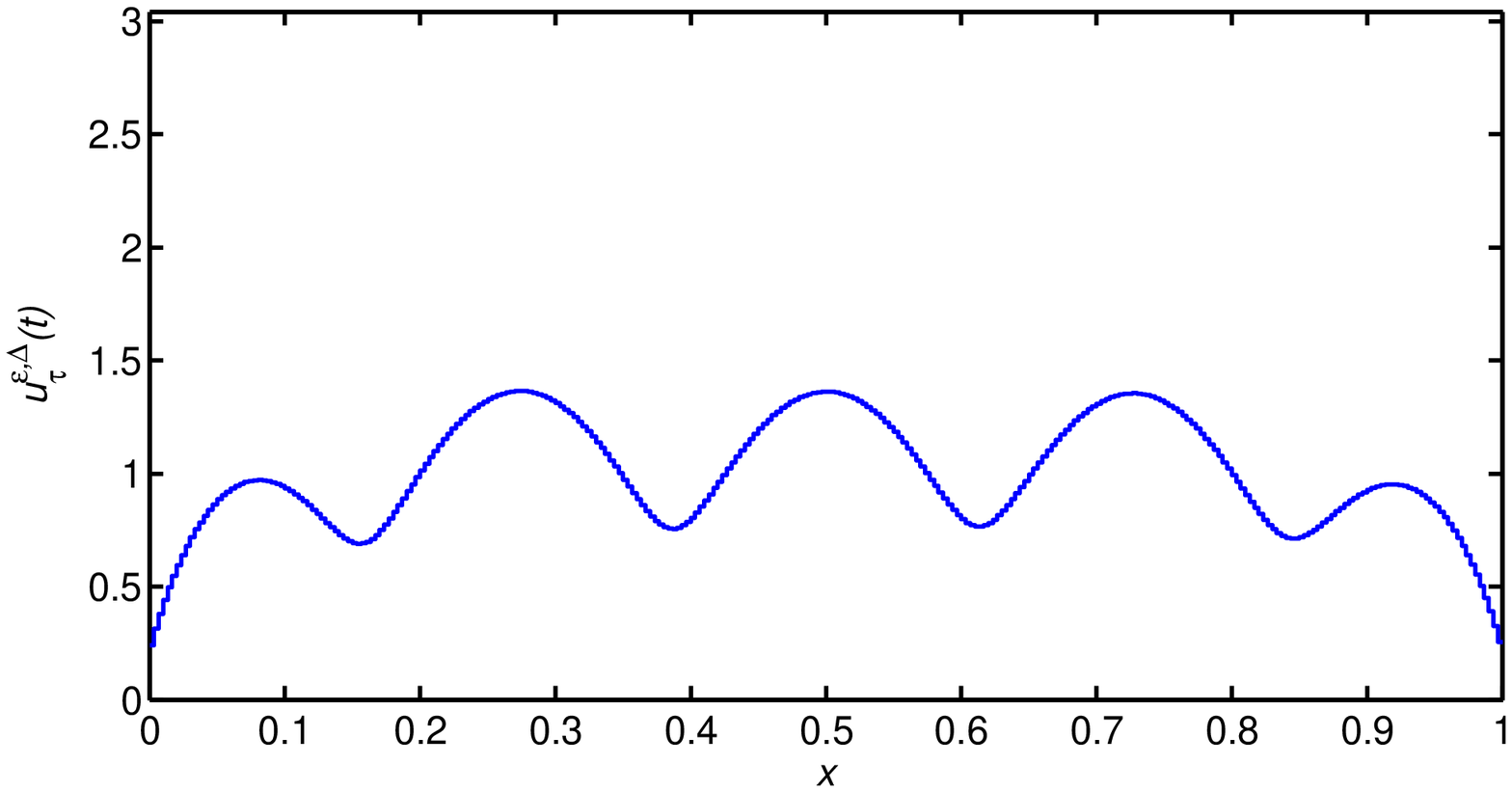}\label{pme:b}}
\subfigure[$t=100\tau=0.01$.]{\includegraphics[width=0.49\textwidth]{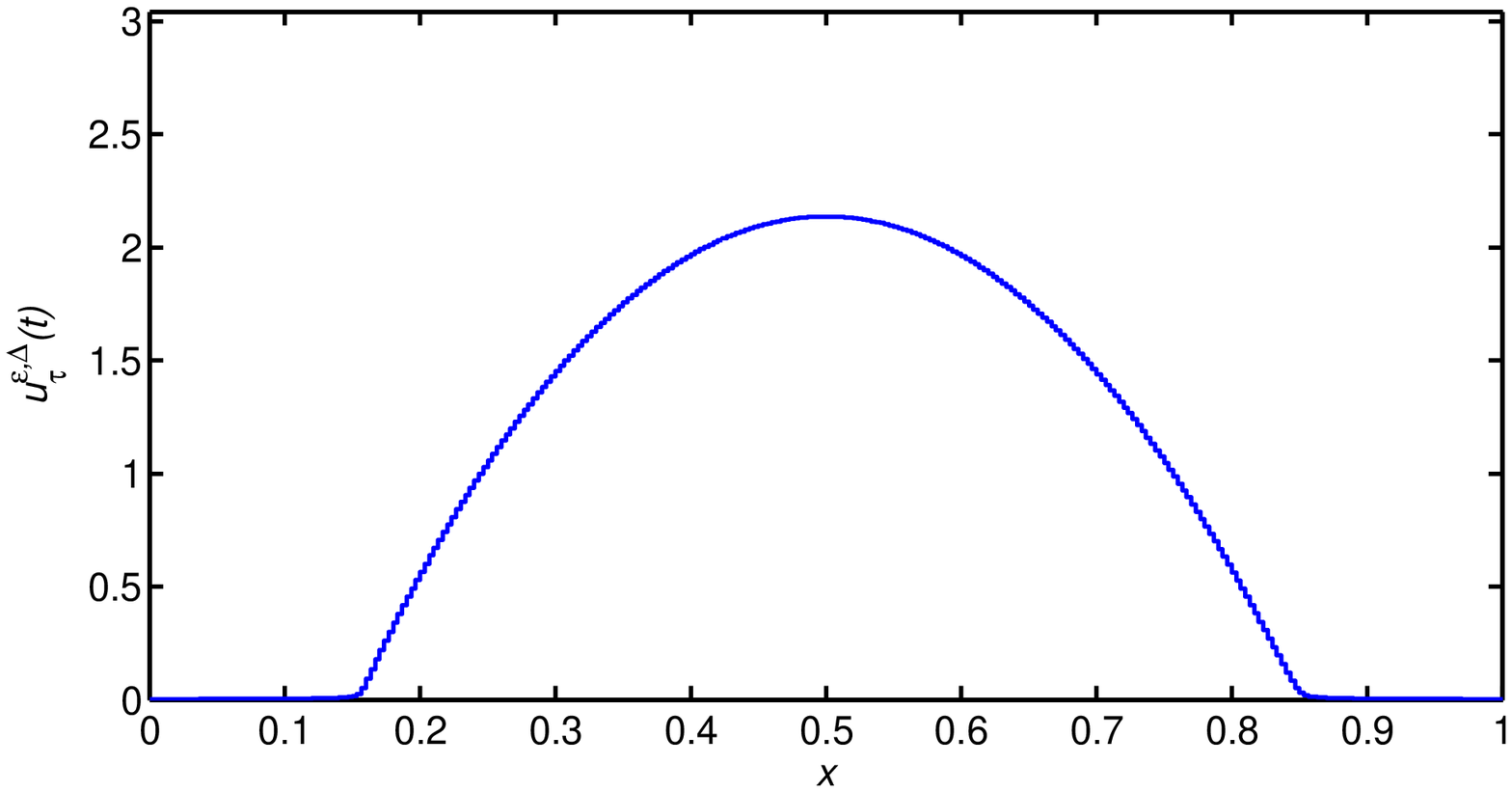}\label{pme:c}}
\subfigure[$t=5000\tau=0.05$.]{\includegraphics[width=0.49\textwidth]{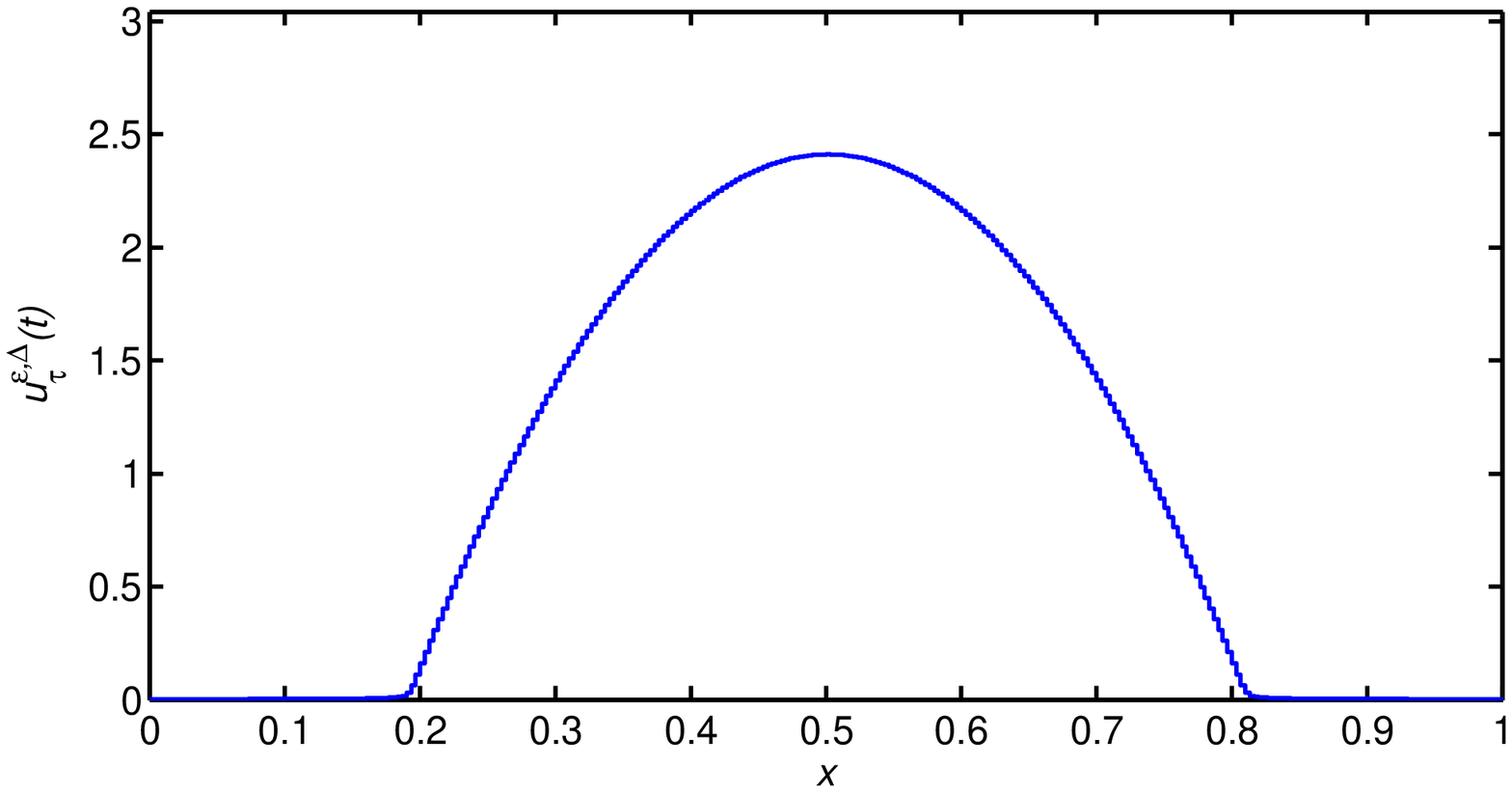}\label{pme:d}}
        \caption{Numerical simulation of $u_\tau^{\eps,\Delta}(t)$ for \eqref{eq:fpe} with $q=2$.} 
        \label{fig:pme}
\end{figure}

Observe that this approximate solution for \eqref{eq:fpe} resembles well the expected analytical solution also incorporating the long-time behaviour \cite{vazquez2007}: as $t\to \infty$, $u_\tau^{\eps,\Delta}(t)$ converges (at an exponential rate) to an approximate version of an (in this case globally stable) equilibrium of \eqref{eq:fpe} which is of Gaussian type for $q=1$ and of Barenblatt-Pattle type for $q>1$ (see Figures \ref{he:d}\&\ref{pme:d}).

\subsection{The Cahn-Hilliard equation}
One of the most interesting evolution equations of fourth order which possess gradient flow structure with respect to the generalized Wasserstein distance $\W_m$ for genuinely nonlinear mobility $m(z)=z(1-z)$ is the \emph{Cahn-Hilliard} equation
\begin{align}\label{eq:cahnhilliard}
\partial_t u=-\theta\partial_x(u(1-u)\partial_x^3 u)+\partial_x(u(1-u)\partial_x (u^2(1-u)^2)),
\end{align}
for $\theta>0$. There, the mobility is such that $M=1$, and the free energy functional is of the form \eqref{eq:E4} with
\begin{align*}
G(p)=\frac{\theta}{2}p^2,\quad E(z)=z^2(1-z)^2\quad\text{and}\quad V\equiv 0.
\end{align*}
In the following, we show simulations of our scheme \eqref{eq:mmsdisc} for different values of the parameter $\theta$. Equation \eqref{eq:cahnhilliard} models the process of phase separation of two components of a binary liquid or alloy; the parameter $\theta$ incorporates the length of the transitions between regions in space where only one of the two components is rich (corresponding to $u\approx 0$ and $u\approx 1$).

Our choice for the discretization parameters is $\Nt=2$, $\Nx=200$ and $\eps=10^{-9}$, and we use the initial condition
\begin{align*}
u_0(x)=0.5(\cos(8\pi x)+1),
\end{align*}
which gives rise to the approximation $u_0^{\eps,\Delta}$ as shown in Figure \ref{fig:chic}.
\begin{figure}[h]
\centering
\includegraphics[width=0.49\textwidth]{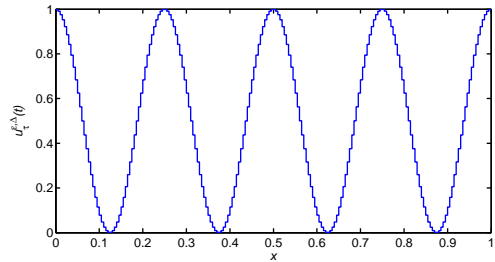}
\caption{Initial approximation $u_0^{\eps,\Delta}$ corresponding to $u_\tau^{\eps,\Delta}(t)$ at $t=0$.}
\label{fig:chic} 
\end{figure}

Figures \ref{fig:ch1} and \ref{fig:ch2} show the spatially discrete function $u_\tau^{\eps,\Delta}(t)$ constructed via the scheme \eqref{eq:mmsdisc} at different time points $t$ for $\theta=0.004$ and $\tau=0.06$ and $\theta=0.001$ and $\tau=0.01$, respectively.
\begin{figure}[H]
\centering
\subfigure[$t=2\tau=0.12$.]{\includegraphics[width=0.49\textwidth]{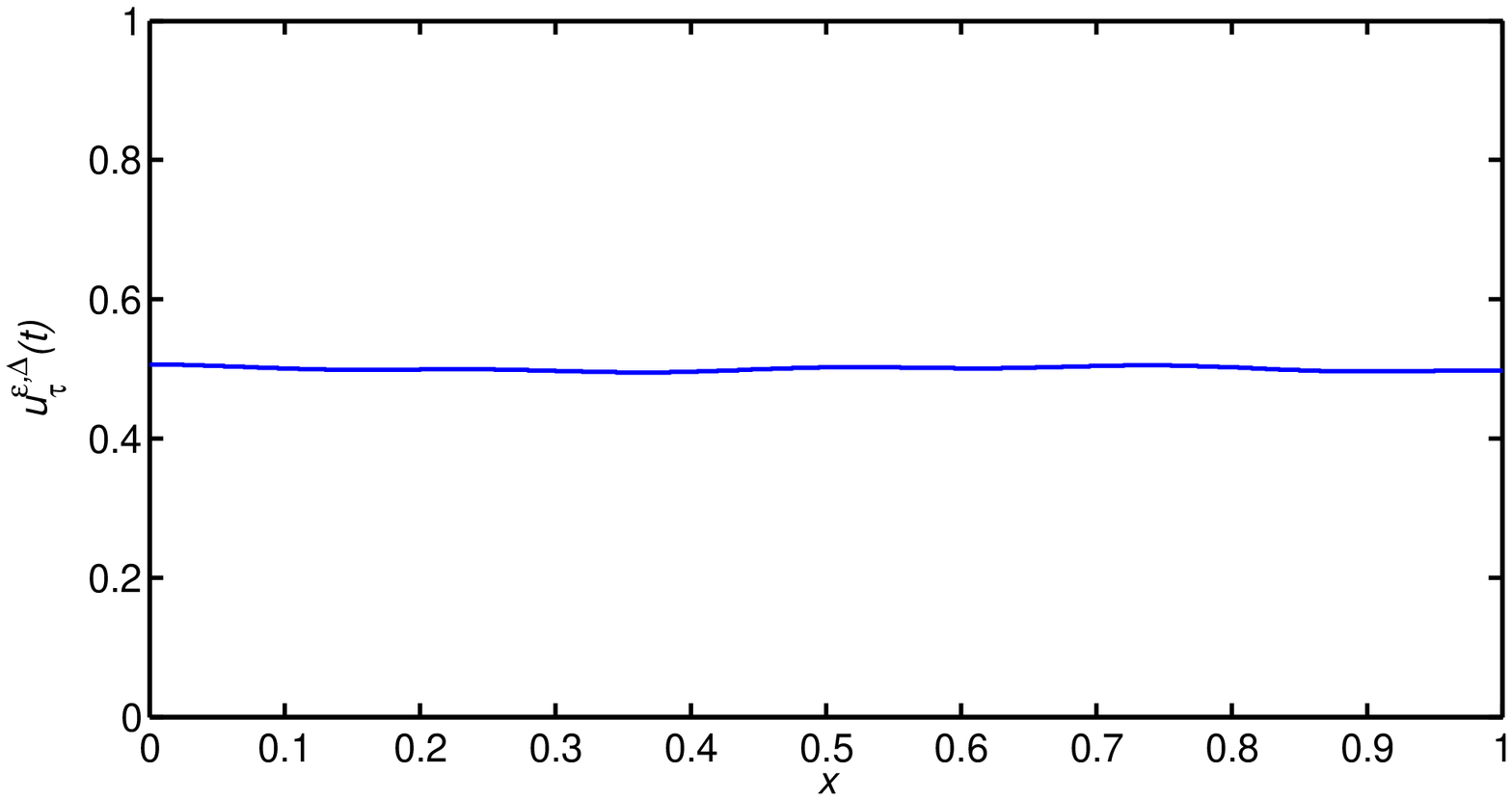}\label{ch1:a}}
\subfigure[$t=100\tau=6$.]{\includegraphics[width=0.49\textwidth]{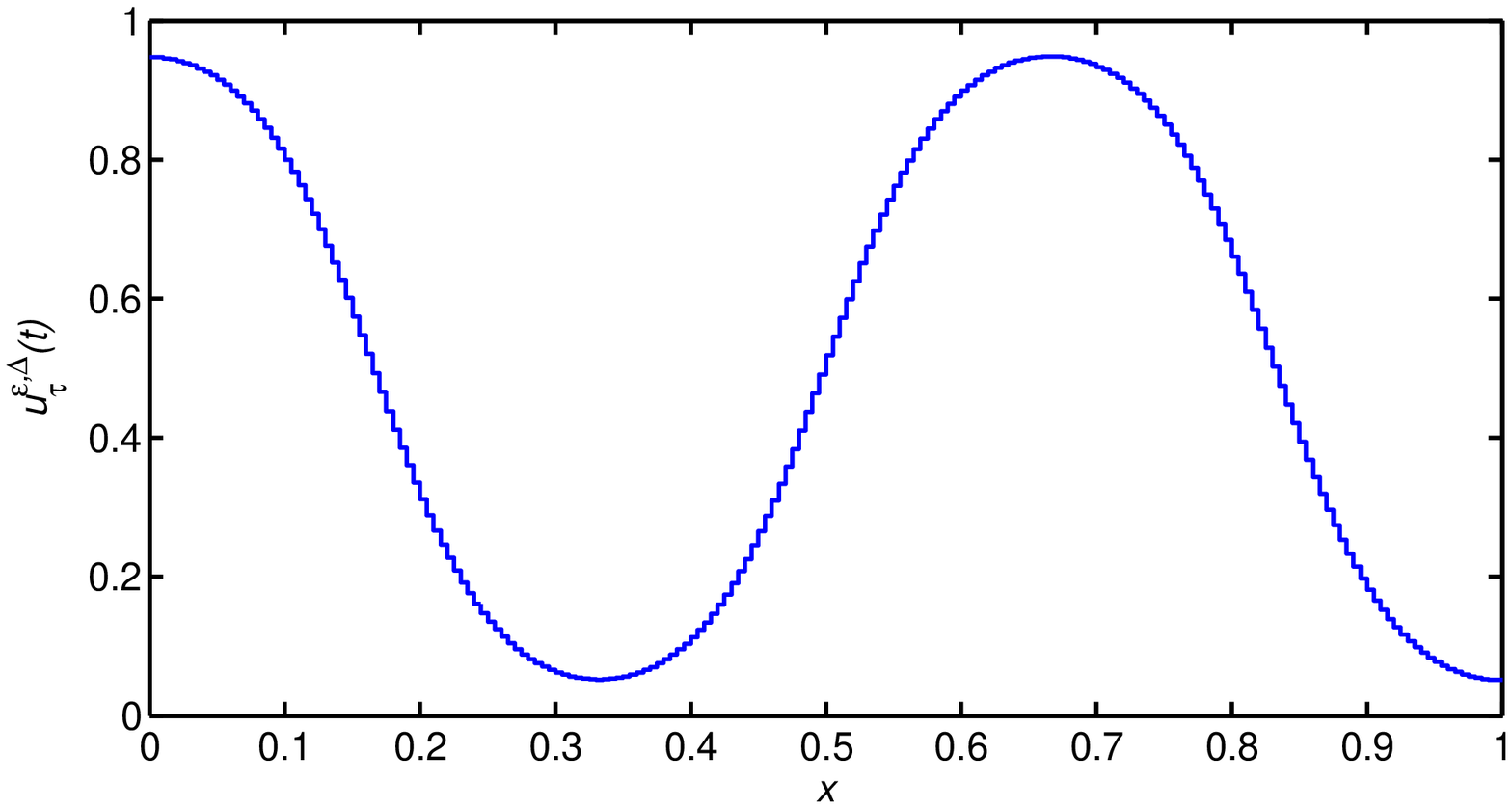}\label{ch1:b}}
\subfigure[$t=2000\tau=120$.]{\includegraphics[width=0.49\textwidth]{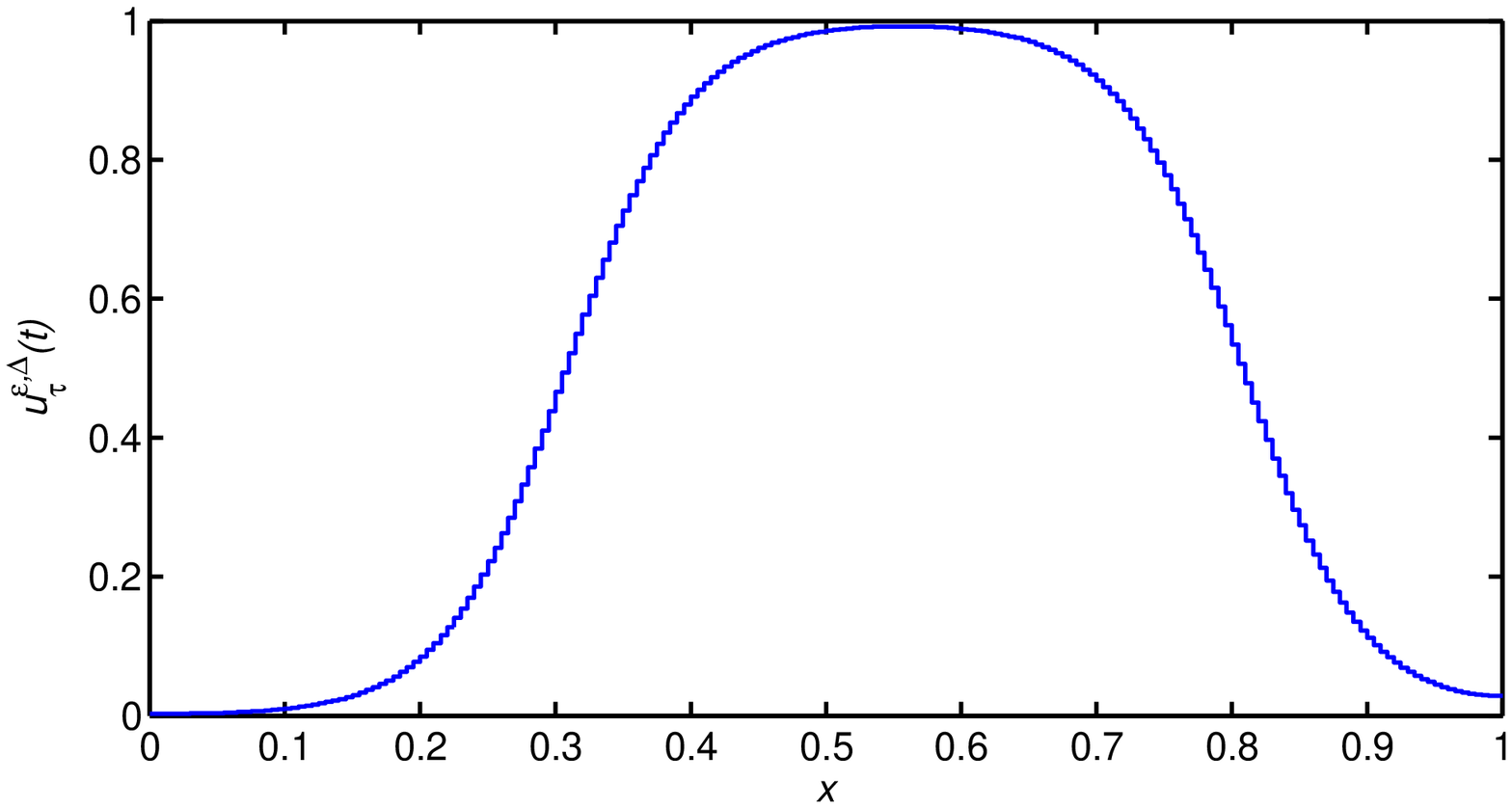}\label{ch1:c}}
\subfigure[$t=11000\tau=660$.]{\includegraphics[width=0.49\textwidth]{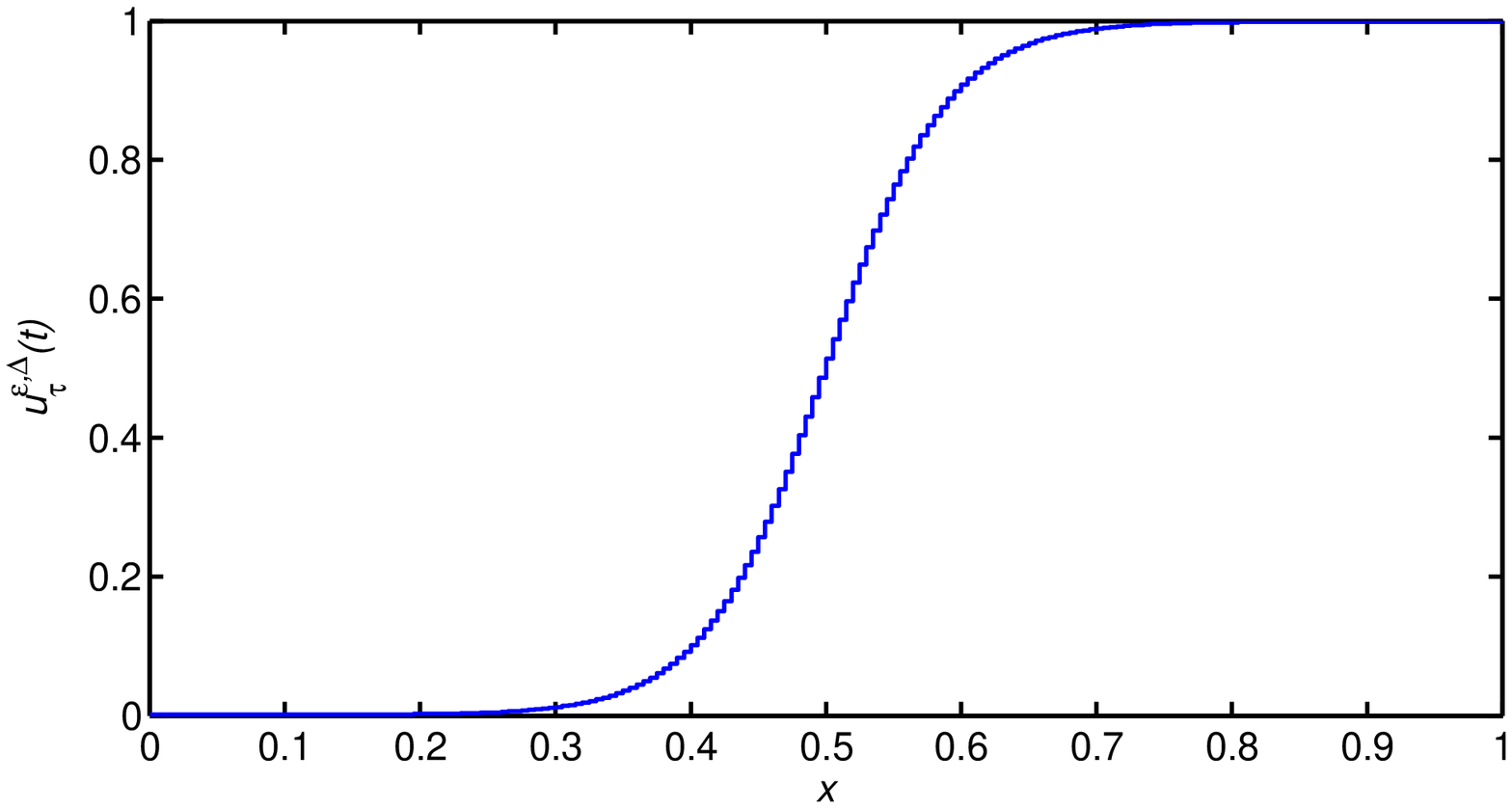}\label{ch1:d}}
        \caption{Numerical simulation of $u_\tau^{\eps,\Delta}(t)$ for \eqref{eq:cahnhilliard} with $\theta=0.004$ and $\tau=0.06$.} 
        \label{fig:ch1}
\end{figure}

\begin{figure}[H]
\centering
\subfigure[$t=50\tau=0.5$.]{\includegraphics[width=0.49\textwidth]{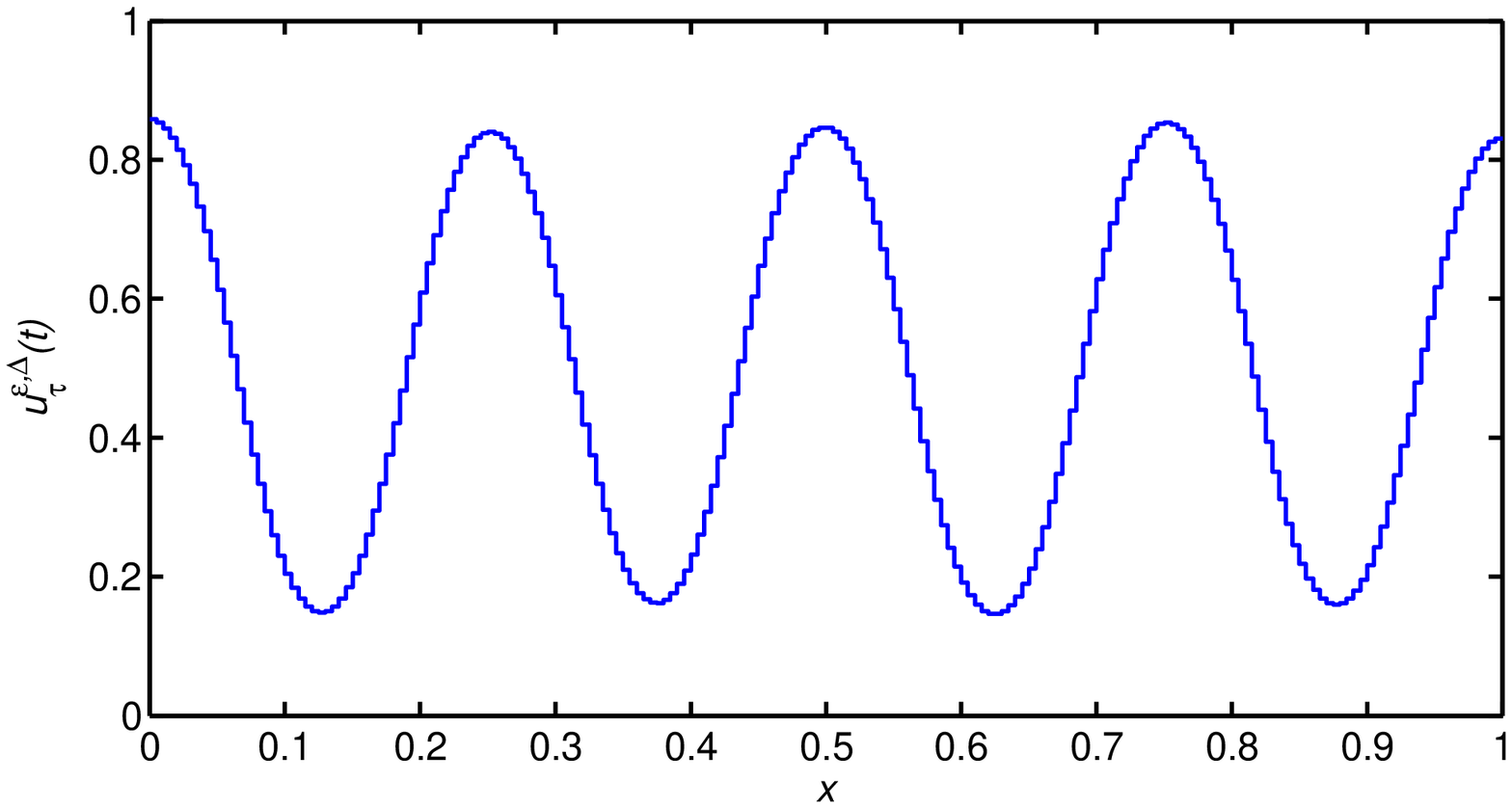}\label{ch2:a}}
\subfigure[$t=100\tau=1$.]{\includegraphics[width=0.49\textwidth]{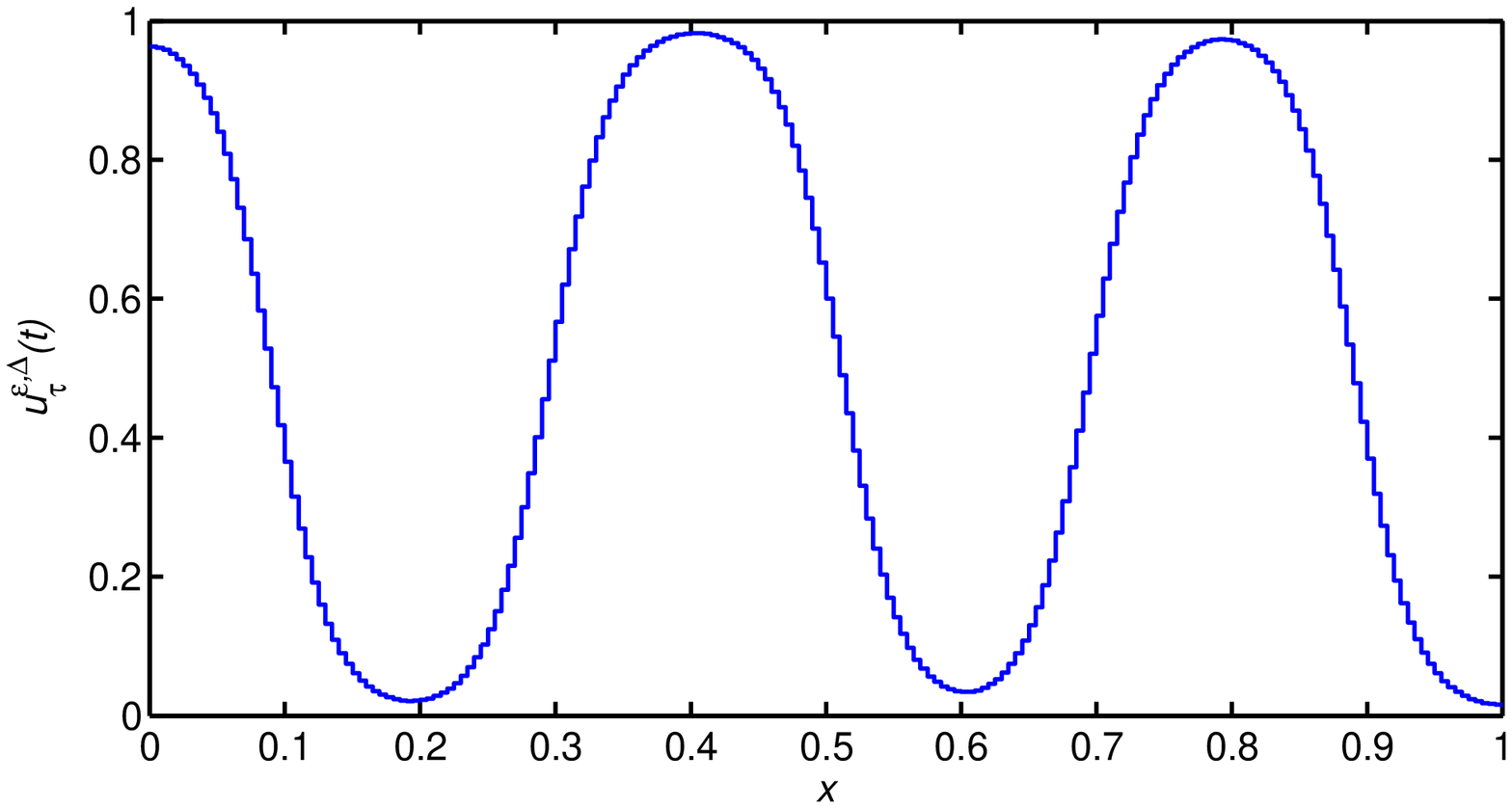}\label{ch2:b}}
\subfigure[$t=3600\tau=36$.]{\includegraphics[width=0.49\textwidth]{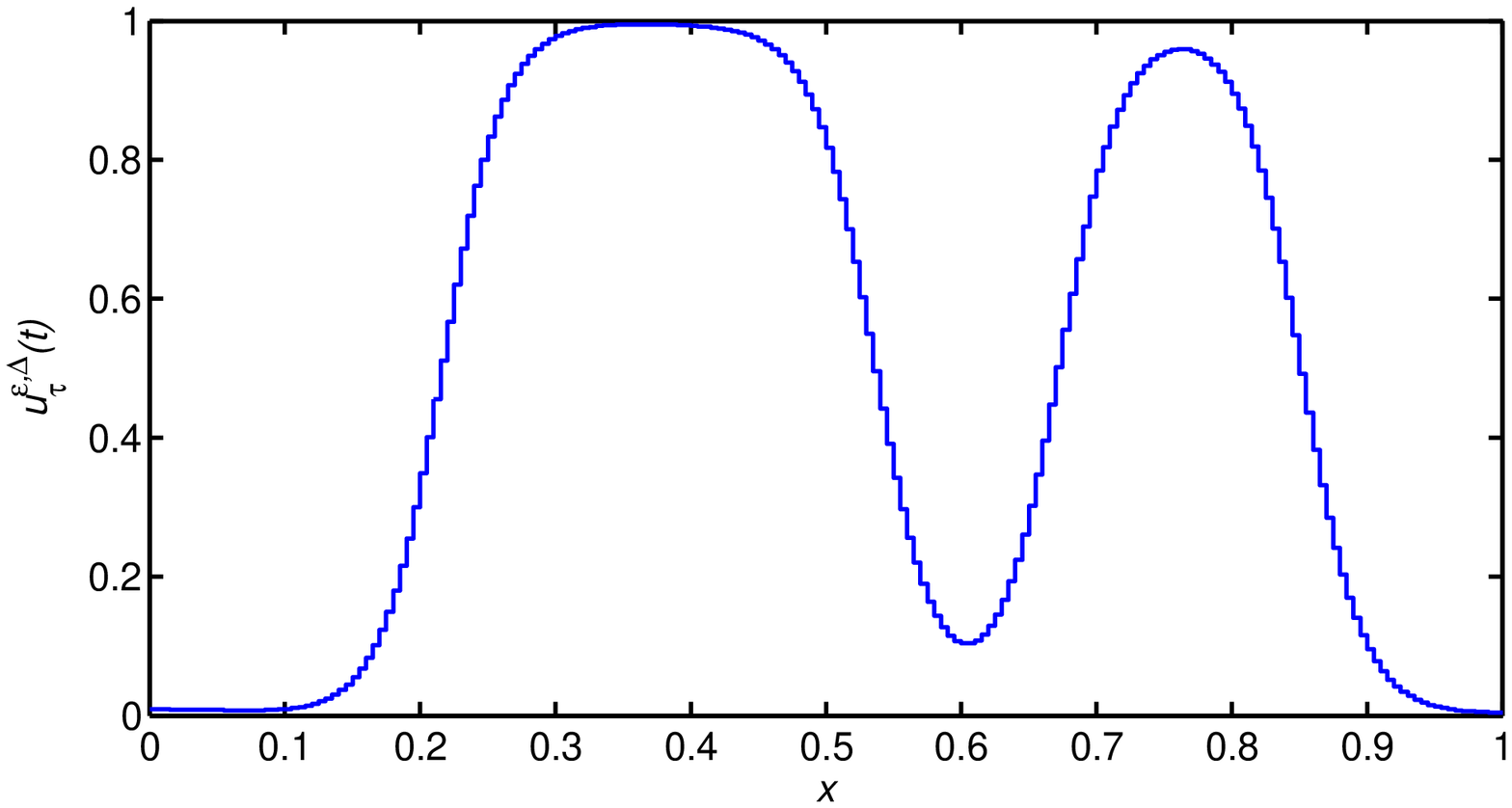}\label{ch2:c}}
\subfigure[$t=10000\tau=100$.]{\includegraphics[width=0.49\textwidth]{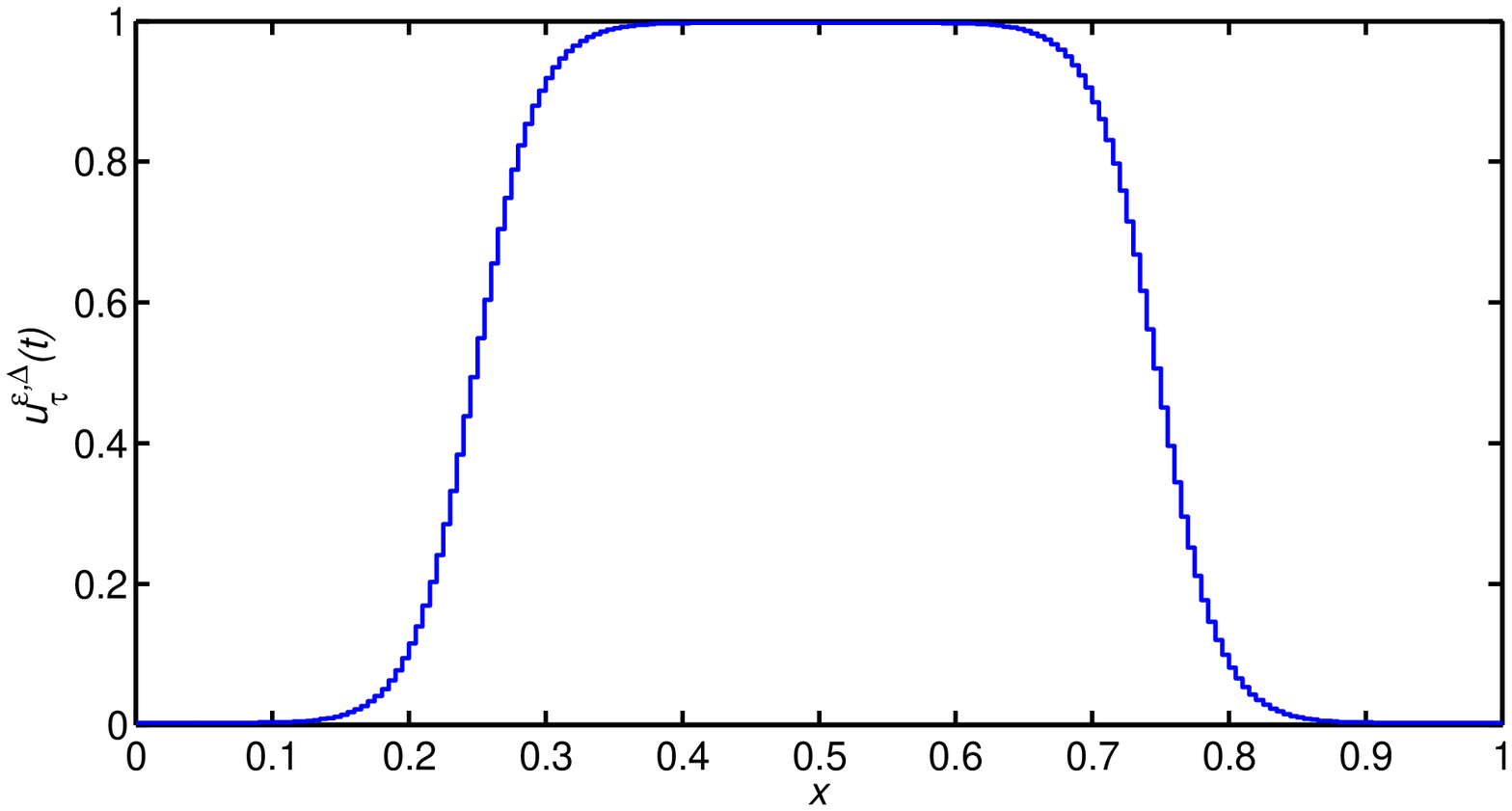}\label{ch2:d}}
        \caption{Numerical simulation of $u_\tau^{\eps,\Delta}(t)$ for \eqref{eq:cahnhilliard} with $\theta=0.001$ and $\tau=0.01$.} 
        \label{fig:ch2}
\end{figure}

The evolution over time corresponds well to the known behaviour of the Cahn-Hilliard equation \cite{elliott1986, zheng1986, elliott1987}: initially (see Figure \ref{ch1:a}), the two components of the fluid mix ($u\approx 0.5$) and then separate into regions where either $u\approx 0$ or $u\approx 1$ (see Figures \ref{ch1:b}\&\ref{ch2:b}). These nonmonotone metastable states change very slowly as regions annihilate one after the other (see Figures \ref{ch1:c}\&\ref{ch2:c}). Eventually, one arrives at a stable monotone steady state of complete separation (see Figure \ref{ch1:d}).

\subsection{The thin film equation}
We conclude this section with another important fourth-order equation: the \emph{thin film} equation
\begin{align}\label{eq:thinfilm}
\partial_t u=-\partial_x (u\partial_x^3 u),
\end{align}
which generates the \emph{Hele-Shaw} flow, and can be interpreted as Wasserstein gradient flow ($m(z)=z$) of the Dirichlet energy $\ent$ with $G(p)=\frac12 p^2$, $E\equiv 0$ and $V\equiv 0$.

Here, we reproduce the numerical results from \cite{becker2005} (see also \cite{osberger2015}) and choose $\Nt=2$, $\Nx=400$, $\eps=10^{-12}$ and $\tau=10^{-5}$, and the initial datum
\begin{align*}
u_0(x)&=(x-0.5)^4+0.001.
\end{align*}

Figure \ref{fig:tf} shows the spatially discrete function $u_\tau^{\eps,\Delta}(t)$ constructed via the scheme \eqref{eq:mmsdisc} at different time points $t$.
\begin{figure}[H]
\centering
\subfigure[$t=0$.]{\includegraphics[width=0.49\textwidth]{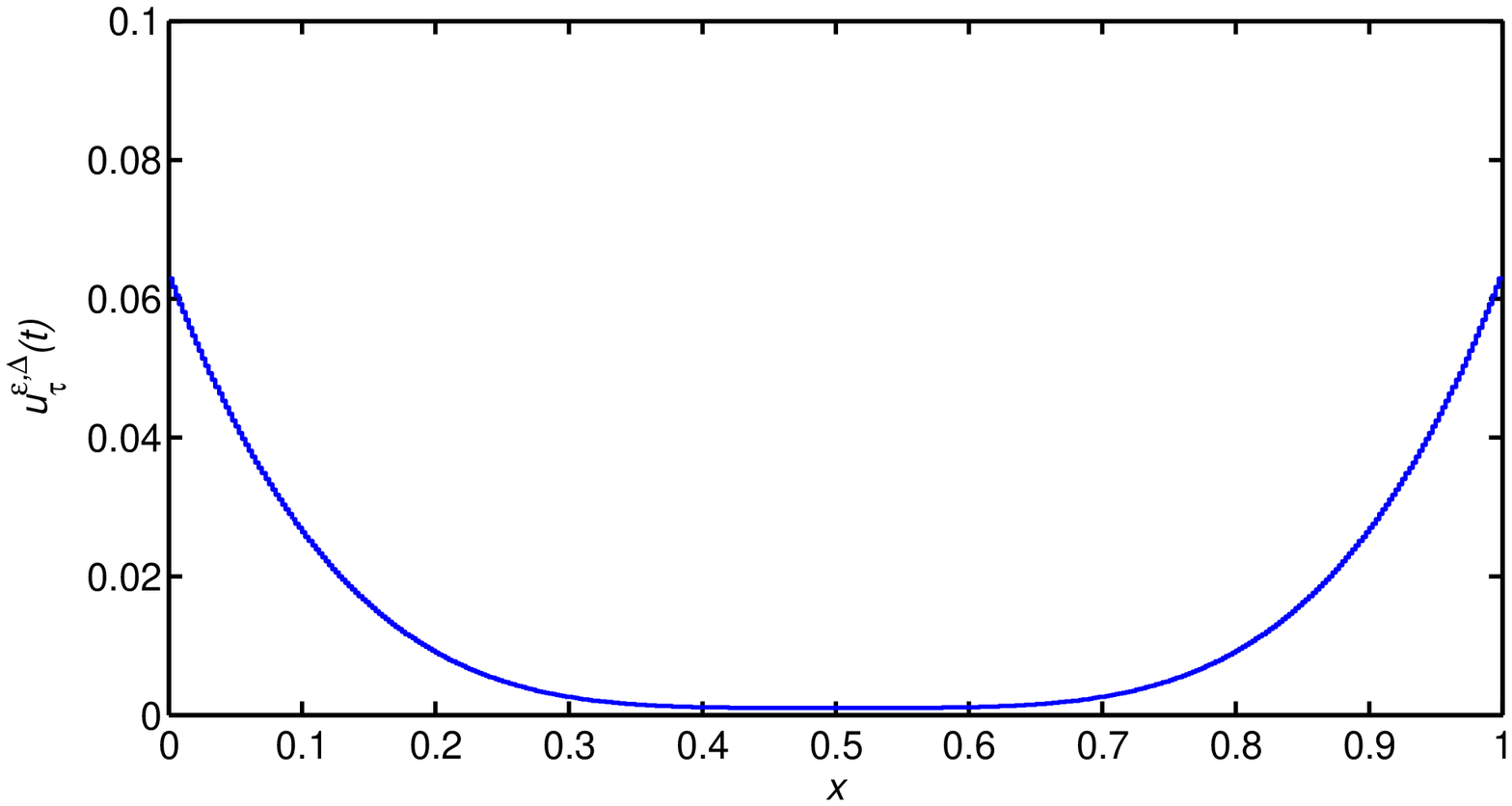}\label{tf:a}}
\subfigure[$t=200\tau=0.002$.]{\includegraphics[width=0.49\textwidth]{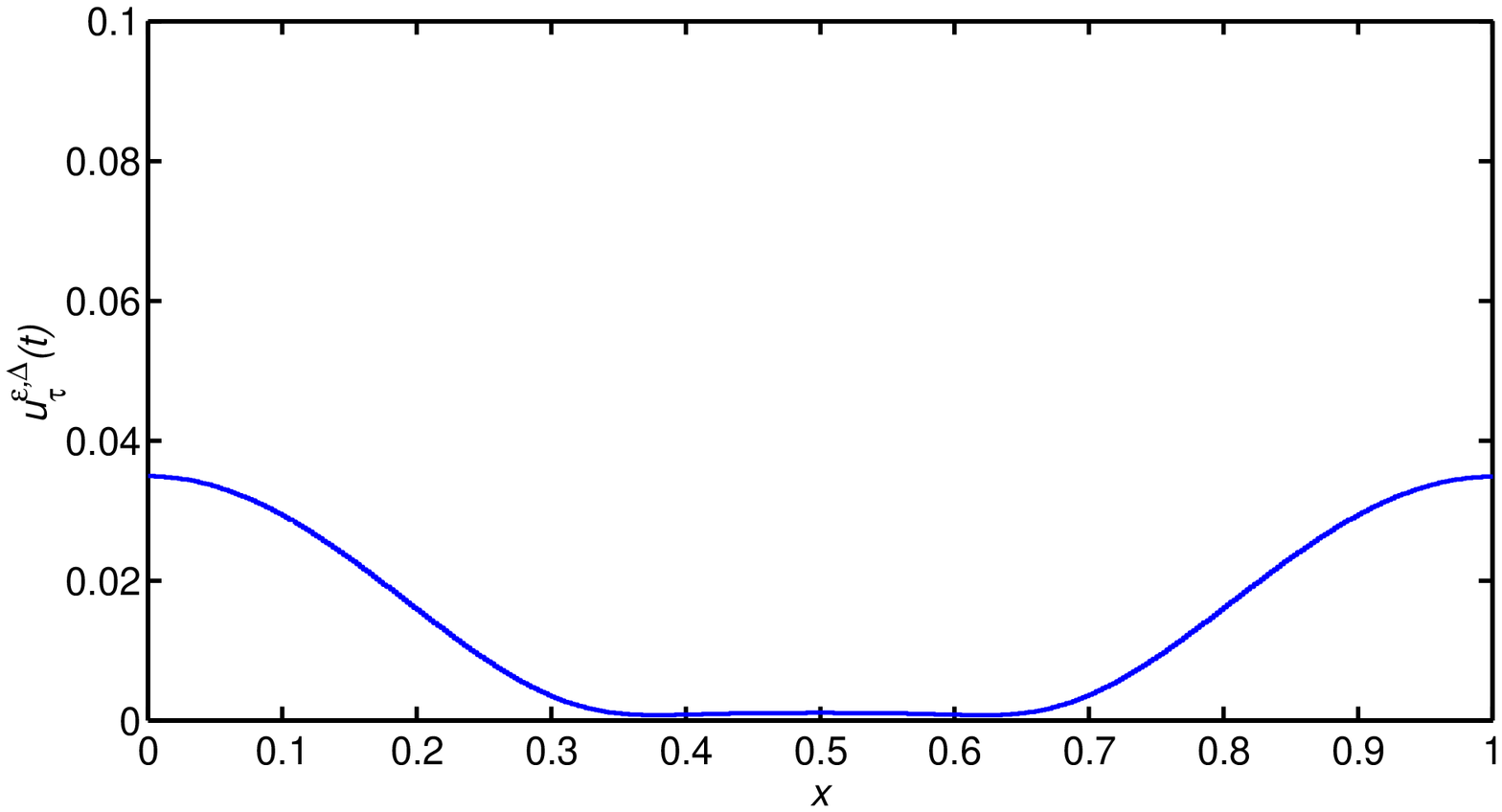}\label{tf:b}}
\subfigure[$t=1200\tau=0.012$.]{\includegraphics[width=0.49\textwidth]{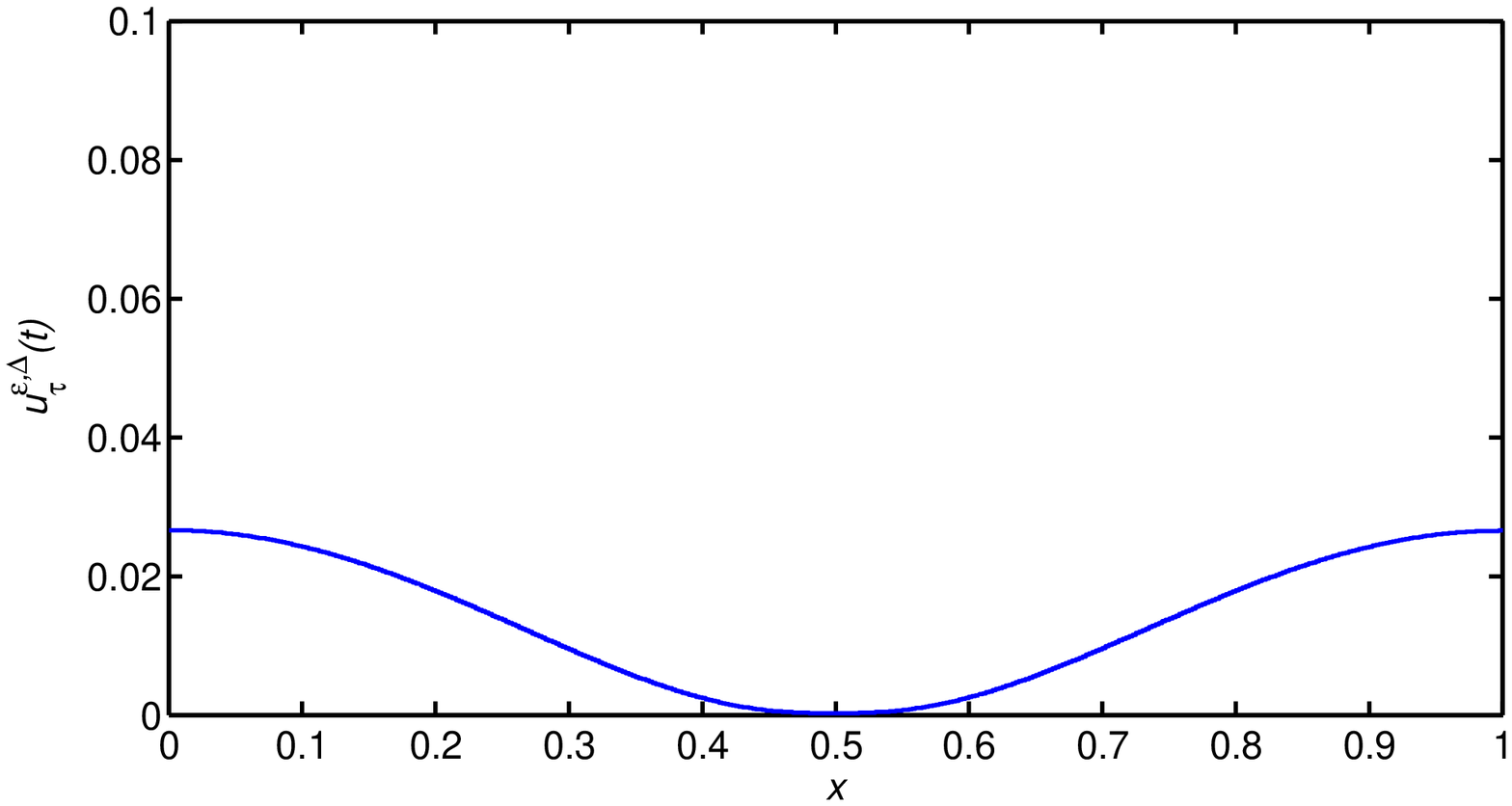}\label{tf:c}}
\subfigure[$t=4000\tau=0.04$.]{\includegraphics[width=0.49\textwidth]{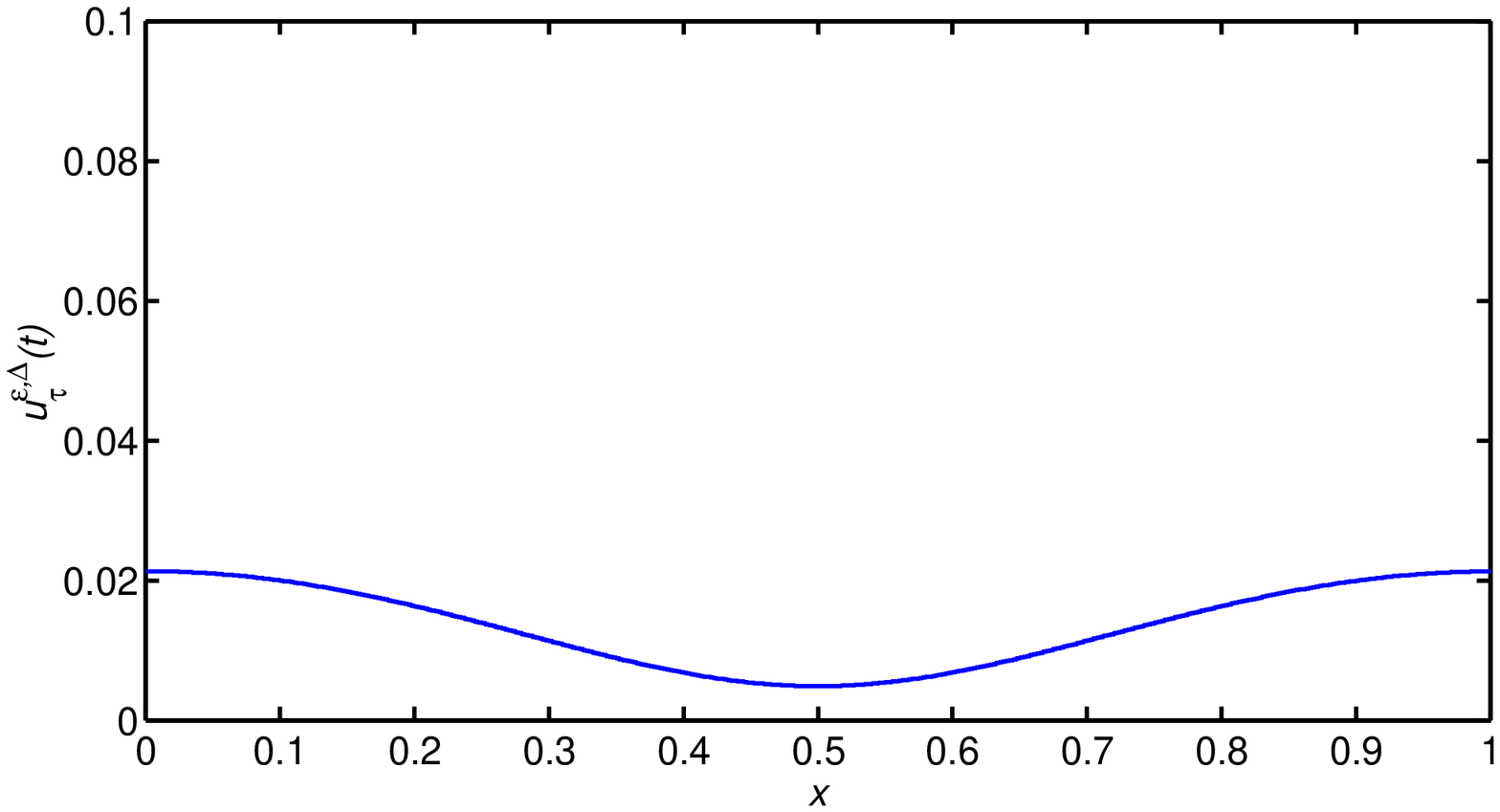}\label{tf:d}}
        \caption{Numerical simulation of $u_\tau^{\eps,\Delta}(t)$ for \eqref{eq:thinfilm}.} 
        \label{fig:tf}
\end{figure}
Our results correspond well to those in \cite{becker2005, osberger2015}: they indicate that at time $t=0.012$ (see Figure \ref{tf:c}), finite-time rupture of the liquid film may occur despite starting with a strictly positive initial datum $u_0$.


\bibliographystyle{abbrv}
   \bibliography{bib}

\end{document}